\documentclass[10pt, reqno]{amsart}
\usepackage{latexsym}
\usepackage{amsfonts, amsmath, amscd}
\usepackage{amssymb}
\usepackage[all]{xypic}
\usepackage{color}
\usepackage{enumerate}

\usepackage{verbatim}
\usepackage{amsthm}    
\usepackage{hyperref}
 \hypersetup{colorlinks=false}

 \usepackage{wrapfig}
 \usepackage{graphicx, psfrag}
   \usepackage{pstool}
 \usepackage[normalem]{ulem}

\def\Bbb{\mathbb}
\def\cal{\mathcal}

\def\al{\alpha}

\def\C{{\cal C}}

\def\dim{\mathrm{dim}}
\def\de{\delta}

\def\text#1{{\em #1}}

\def\la{\lambda}

\def\ep{\varepsilon}

\def\si{\sigma}
\def\om{\omega}
\def\be{\begin{equation}}
\def\ee{\end{equation}}
\def\bear{\begin{eqnarray}}
\def\eear{\end{eqnarray}}
\def\best{\begin{eqnarray*}}
\def\eest{\end{eqnarray*}}

\renewcommand{\theequation}{\arabic{section}.\arabic{equation}}
\newtheorem{theorem}{Theorem}[section]
\newtheorem{prop}[theorem]{Proposition}

\newtheorem{lemma}[theorem]{Lemma}
\newtheorem{cor}[theorem]{Corollary}

\newtheorem{defn}[theorem]{Definition}

\newtheorem{assumption}[theorem]{Assumption}
\newtheorem{remark}[theorem]{Remark}
\newenvironment{rem}{\begin{remark}\rm}{\end{remark}}

\newtheorem{example}[theorem]{Example}
\newenvironment{ex}{\begin{example}\rm}{\end{example}}

\def\non{\noindent}
 \def\pf{\non {\it Proof. }}
\def\qed{\nopagebreak \hskip .1in { $\Box$ }\penalty10000 %
\hskip\parfillskip \par  }

\def\ra{\rightarrow}

\def\r#1{\right#1}
\def\l#1{\left#1}

\def\ma#1{\mathop {#1} \limits}

\def\De{\Delta}
\def\phi{\varphi}
\def\ti{\times}

\def\del{\overline \partial}

\def\Z{{ \Bbb Z}}
\def\R{{ \Bbb R}}
\def\P{{ \Bbb P}}
\def\Q{{ \Bbb Q}}
\def\cx{{ \Bbb C}}

\def\J{\cal J}
\def\JV{\cal{JV}}

\def\M{{\cal M}}
\def\oM{\overline{\cal M}}

\def\U{{\cal U}}

\def\N{{\cal N}}

\def\ev{\mathrm{ev}}
\def\st{\mathrm{st}}
\def\ct{\mathrm{ct}}
\def\pr{\mathrm{pr}}
\def\pr{\mathrm{pr}}

\def\wt#1{\widetilde{#1}}
\def\wh#1{\widehat{#1}}
\def\ov#1{\overline{#1}}

\def\w{\omega}

\def\Aut{\mathrm{Aut}}

\def\Hom{\mathrm{Hom}}

\def\longra{\longrightarrow}

\def\Mu{\boldsymbol{\mu}}

\def\Sym{\mathrm{Sym}}

\def\Ev{\mathrm{Ev}}

 \newcommand\cHH{\check{\mathrm{H}}}  
\newcommand\Cech{\v{C}ech\ }

\def\longra{\longrightarrow}
\def\cHH{\check{H}}
\def\al{\alpha}

\def\tC{\tilde{C}}
\def\Sym{\mathrm{Sym}}

 \usepackage{hyperref}
 \hypersetup{colorlinks=false}

 \def\Cech{\v{C}ech\ }

\begin{document}

\title{\bf  A natural Gromov-Witten virtual fundamental class     \vskip.2in}
\author{Eleny-Nicoleta Ionel}
\address{Stanford University,
 Palo Alto, CA, USA.}
\email{ionel@math.stanford.edu}

\author{Thomas H. Parker}
\address{Michigan State University, East Lansing MI, USA. }
\email{parker@math.msu.edu}

\maketitle



 \begin{abstract}
We describe a program for proving that the Gromov-Witten moduli spaces of compact symplectic manifolds carry a unique virtual fundamental class that satisfies certain naturality conditions.     The virtual fundamental class is constructed using only Ruan-Tian perturbations by   introducing stabilizing divisors, using  \v{C}ech homology,  and systematically applying  naturality conditions.  In high dimensions or low genus, no gluing theorems are needed.
\end{abstract}


\setcounter{equation}{0}
\section{Introduction}
\label{section1}
\bigskip

The central object of Gromov-Witten theory is the virtual fundamental class (VFC) of the moduli space of stable maps.  In the algebraic geometry context, the virtual fundamental class was constructed soon after the subject began.  It is difficult to extend the construction to the symplectic category, in large part because the methods of geometric analysis are ill-suited to apply to maps from families of degenerating curves.  
The issue was originally addressed in the late 1990s (\cite{fo}, \cite{LT}, \cite{liu-t},  \cite{R}, \cite{siebert}).  As the subject  has developed, there has been an increasing need for an approach that is conceptual and also easy to use.   In recent years  major efforts have been made to  define a   virtual fundamental class  using Kuranishi structures  \cite{fooo}, \cite{joyce},  \cite{mw},  and  polyfolds \cite{hofer}.  Recently, J.~Pardon has developed a very nice approach using implicit atlases \cite{pardon}.

The program described in this article is an alternative approach,  building on ideas of Cieliebak-Mohnke \cite{cm} and guided by the functorial aspects of the VFC.   Parts of this program are currently conjectural. The strategy  is to first observe that existing results give a  virtual fundamental class  for a set of  especially nice almost complex structures  (which we call super-fine and regular),  and  then successively enlarge the space of parameters on which the VFC is defined. This strategy works well in high dimensions ($\dim X\ge 12$) or low genus ($g\ge 1$) where, as observed by Tehrani and Zinger in \cite{tz}, dimension counts obviate any need for gluing theorems.  (We do not address the lower dimensional case here.) To fully exploit functoriality, we  consider a variety of  related Gromov-Witten theories.

\medskip

 Here is the general setting.   In  symplectic Gromov-Witten theory, each moduli space is part of a  family parameterized by  the space $\J(X)$ of  almost complex structures on $X$ that are tamed by a fixed symplectic structure $\w$, or more generally, the space of all tamed pairs $(J,\w)$.  One must also specify the type of domain curve and type of target.  The choices involved can be phrased as follows.

A  ``GW moduli problem'' for a symplectic manifold  $X$, or more generally a symplectic pair $(X,V)$,  consists of
\begin{enumerate}
\item[(a)]  A choice of a category of families of complex curves, 
\item[(b)]  A choice of a (possibly empty) normal crossing divisor $V\subset X$, 
\item[(c)]  A choice of a  set $\J$ of that parameterizes a set of pseudo-holomorphic maps.  
\end{enumerate}

From this data one constructs a GW family  $\oM(X)$ of moduli spaces over $\J$ whose fiber $\oM^J(X)$ over $J\in\J$ is the (compact) set of all isomorphism classes of $J$-holomorphic maps from curves of the specified type.  We will assume that (a) is a category of curves with   $n$ marked points and a Deligne-Mumford space $\ov{\M}_n$; then there is an stabilization-evaluation map for the family:
\bear
\label{0.1}
\xymatrix{
\oM(X) \ar[r]^{\quad  se  \qquad } \ar[d]_\pi & \oM_n\times X^n\\
\J & 
}
\eear

\noindent {\bf Meta-Theorem}:  Every fiber of $\pi$ carries a virtual fundamental class 
$[\oM^J(X)]^{vir}\in \cHH_*(\oM^J(X))$ in rational \Cech homology.\\


The existence of a virtual fundamental class immediately gives the Gromov-Witten 
invariants by pushing forward by the $se$ map.  
\medskip


The purpose of  our program  is to make this meta-theorem precise,  to outline a proof, and to describe some functorial and uniqueness properties of the virtual fundamental class.

\medskip

The following examples emphasize the fact that there is not one, but many GW theories, depending on the choices (a)--(c).

\begin{ex} It is standard to consider the compact moduli spaces $\oM^J_{A, g, n}(X)$ of stable $J$-holomorphic maps from a genus $g$ $n$-marked curve to $X$ that represent $A\in H_2(X)$. These give a diagram (\ref{0.1}) where
\best
se: \ov\M_{A, g, n}(X)\ra \ov\M_{g, n} \ti X^n.
\eest 
The $GW$ invariants are 
\best
GW_{A, g, n}(X)=se_*[\ov\M_{A, g, n}^J(X)]^{vir}\in H_*(\oM_{g, n} \ti X^n).
\eest 
Each $\oM_{A, g, n}(X)$ is an open and closed subset of $\oM(X)=\ma\sqcup_{A, g, n}\oM_{A, g, n}(X)$.
 \end{ex}

\begin{ex}
\label{ex.add.G} 
One can also consider maps with decorated domains.  One example, described in Section~2, is obtained by fixing a finite group $G$ and considering  curves with twisted $G$-covers. 
 Again, there is a space of stable maps and  a stabilization-evaluation map  
\best
se: \ov\M^G_{A, g, n}(X)\ra \ov\M^G_{g, n} \ti X^n. 
\eest
\end{ex}
\begin{ex}  One can also decorate the target by choosing a normal crossing divisor $V\subset X$ and making two modifications:  a) restrict the parameter space  $\J$  to the subspace  $\J(X,V)$ of $V$-compatible almost complex structures, 
and b) mark all the points in the inverse image of $V$, and decorate them with a vector $s$ that records  contact multiplicities to the various branches of $V$.  Using the relative stable map compactification, one obtains moduli spaces $\oM(X,V)$ of relatively stable maps and a refined  $se$ map 
\best
se: \ov\M_{A, g, n, s}(X, V)\ra \ov\M_{g, n+\ell(s)} \ti X^n\ti \P_s (NV)
\eest
as described in Section~\ref{section7}. There is also a $G$-twisted version
\best
se: \ov\M_{A, g, n, s}^G(X, V)\ra \ov\M_{g, n+\ell(s)}^G \ti X^n\ti \P_s (NV).
\eest
\end{ex}


To include all these examples, we consider families of moduli spaces    
 \bear
\label{1.ms}
\xymatrix{
\ov\M_{A, g, n, s}^{G,P}(X, V)\ar[d]_\pi \\
\ \qquad P \subset \J
}
\eear
parameterized by some subset $P$ of almost complex structures or Ruan-Tian perturbations (see Section~3.1), and where $\oM^{G}(X, V)$ denotes the moduli space of $G$-twisted maps into $(X, V)$, for a finite group $G$ and a normal crossing divisor $V$ in $X$. 
 Each such family of moduli spaces has a natural stabilization-evaluation map $se$ and  an expected dimension $d$ depending on $(A, g, n, s)$.  These families are related by   functorial maps of several types:
\begin{enumerate}[(i)]
\item An inclusion $\iota:  Q\hookrightarrow  R$ of compact subsets of $\J$  yields   a diagram
\bear
\label{iota}
\xymatrix{
\ov\M^{G,Q}_{A,g,n,s}(X,V)\ar[d]\ar[r]^{\qquad \iota \qquad}&\ \ov\M_{A,g,n,s }^{G,R}(X,V) \ar[d] \\
Q\ar[r]^\iota &R
}
\eear
\item  Forgetting the twisted $G$-cover induces a map 
\bear
\label{1.phiG}
\phi_G: \ov\M^G_{A,g,n,s}(X,V)\to \ov\M_{A,g,n,s}(X,V).
\eear 
over the subspace $\J(X, V)$ of parameters that do not depend on the $G$-structure.
\item  Forgetting a smooth divisor $V$ with the corresponding contact points induces a map 
\bear
\label{1.phiVmapabsolute}
\phi_{V}: \ov\M_{A,g,n, s}(X,V)  \to \ov\M_{A,g,n}(X)
\eear
over the space of parameters $\J(X,V)$ that do not depend on the $\ell(s)$ contact points. Similarly,  forgetting a component $V$ of a normal crossing divisor $V\cup W$  induces a map  
\bear
\label{1.phiVmap}
\phi_{W}: \ov\M_{A,g,n, s}(X,V\cup W)  \to \ov\M_{A,g,n,s_W}(X,W)
\eear
over $\J(X, V\cup W)$, where $s_W$ is the subset of $s$ that records the intersection multiplicities with $W$.

\end{enumerate}



\bigskip

A parameter $J\in P$  is called  {\em regular} if the linearization of the $J$-holomorphic map equation on each stratum is a surjective operator,  for every $J$-holomorphic map $f\in \ov\M^{J}_{A,g,n,s}(X,V)$.   Two classes of regular $J$ are especially easy to work with:  a regular $J$ is {\em domain-stable}  if the domain of every $f$ in the moduli space is a stable curve,  and is   {\em domain-fine}  if each domain is a stable curve with no non-trivial automorphisms.

  Ideally, one could hope to find a parameter space $P$ in  which the fiber $\oM^J$ over each regular $J$ is a manifold of the expected dimension $d$.  For such $J$, one expects that the restriction of (\ref{1.phiG})  to the fiber to be a map
 $$
\phi_G: \ov\M^{G,J}_{A,g,n,s}(X,V)\to \ov\M^J_{A,g,n,s}(X,V)
$$ 
of  degree equal to the degree $\deg_G$ of the finite branch cover $\phi_G: \oM^G_{g, n+\ell(s)} \ra \oM_{g, n+\ell(s)}$. 

  Focusing on the case where $s$ is ${\bf 1}=(1, \dots, 1)$, one might guess  that for a smooth divisor $D$ constructed by Donaldson's Theorem,  \eqref{1.phiVmapabsolute}  is a  map
$$
\phi_{D}: \ov\M^J_{A,g,n, {\bf 1}}(X,D)  \to \ov\M^J_{A,g,n,{\bf 1}}(X)
$$
of degree  $\ell(s)!=(A\cdot D)!$.   In Section~11, we show that this property holds for a non-empty class of ``sufficiently positive  divisors''  $D$ (cf. Definition~\ref{def8.1}) provided that either $\dim_\R\, X\ge 10$ or $g\le 1$  (outside this range, there may be  correction terms coming from maps with components in $D$.)

Similarly, one might conjecture  that \eqref{1.phiVmap} restricts to be a map 
$$
\phi_{D}: \ov\M^J_{A,g,n, s, {\bf 1}}(X,V\cup D)  \to \ov\M^J_{A,g,n,s}(X,V),
$$
also of degree $(A\cdot D)!$.  In Section~11 we show that this property  holds if the branch $D$ belongs to  a more restricted, again non-empty,  class of sufficiently positive  divisors, and either   $g\le 1$ or every stratum of $V\cup D$ is at least 8 dimensional.

\bigskip

 Each subset $Q$ of the parameter space $P$ in \eqref{1.ms} determines a moduli space $\oM^Q(X)=\pi^{-1}(Q)$ over $Q$.  
As shown in Section~\ref{section4},   the  family (\ref{1.ms}) has a metric space topology making $\pi$ a proper map to  $P$  with its $C^0$ topology.   One can then pass to  \v{C}ech homology with rational coefficients, which defined for all metric spaces and is equal to Steenrod homology with rational coefficients for all compact metric spaces. For facts about these homology theories we refer the reader to \cite{ma1}, \cite{ma2} and  \cite{milnor}.   In  \Cech homology,  a compactification $\ov{M}=M$ of an oriented  topological manifold $M$ has a fundamental class $[\ov{M}]$ provided that the set $S=\ov{M}\setminus  M$ of ``boundary strata'' has codimension~2, regardless of how these strata fit together (see \cite[page 357]{DK}, and \cite{IPFC} for details).


\begin{defn}
\label{1.defnVFC}
{\rm A {\em virtual fundamental class}  for a  moduli space $\oM^P=\oM_{A,g,n,s}^P(X,V)$  over a path-connected manifold $P\subset \J$ associates  to each compact subset  $Q\subseteq P$  an element 
 $$
 [\oM^Q]^{vir}\in \cHH_*(\oM^Q; \Q)
 $$
such that a regularity axiom and three naturality axioms hold:
\begin{itemize} \setlength\itemsep{4pt}
\item[{\bf A1.}] If $J\in P$ is regular then $[\oM^{J}]^{vir}$ is the fundamental class $[\oM^{J}]$. 
\item[{\bf A2.}] A proper inclusion $\iota:Q\to R$ induces $\iota_*[\oM^Q]^{vir}\ =\ [\oM^R]^{vir}$. 
\item[{\bf A3.}] Under \eqref{1.phiG},    \    $(\phi_G)_*[\oM^{G,Q}]^{vir}\ =\deg_G \cdot [\oM^Q]^{vir}$. 
\item[{\bf A4.}]  Suppose that $A\not= 0$, that  $D$ is a  smooth  sufficiently positive divisor,  
that $Q$ is a subset of $\J(X, D)\subset \J(X)$,  and that  $\dim_\R X\ge 10$ or $g\le 1$.   Then under \eqref{1.phiVmapabsolute} with $s = {\bf 1} =(1,\dots, 1)$,  
$$
 \frac{1}{\ell(s)!}\, (\phi_D)_*[\oM_{A,g,n,{\bf 1}}^Q(X,D)]^{vir}
\ =\  [\oM_{A,g,n}^Q(X)]^{vir}.
$$
More generally, suppose that $A\not= 0$,  $V\cup D$ is a normal crossing extension of $V$ with $D$ sufficiently positive,  $Q\subset \J(X, V\cup D)$,  and that either $g\le 1$ or all strata of $V\cup D$ are at least 8 dimensional. Then under \eqref{1.phiVmap} and  $s_D= {\bf 1} =(1,\dots, 1)$, 
 $$
 \frac{1}{\ell(s_D)!}\, (\phi_D)_*[\oM_{A,g,n,s,{\bf 1}}^Q(X,V\cup D)]^{vir}
\ =\  [\oM_{A,g,n,s}^Q(X,V)]^{vir}.
$$
\end{itemize}
}
\end{defn}
\bigskip

With this terminology, our main result can be stated simply:

\begin{theorem}
\label{VFCfinecase} Assume $X$ is a closed symplectic manifold and $V$ a normal crossing divisor in $X$. Whenever $A\ne 0$ and  
\bear\label{As*}
\mbox{either $g\le 1$, or  $X$ and every stratum of $V$ are at least 12 dimensional}
\eear
then 
there is a unique virtual fundamental class associated to  $\ov\M_{A,g,n, s}(X, V)\to \J(X, V)$.  In particular, 
\begin{enumerate}
\item[(a)]   For each $J\in \J(X)$,  there is a class
\best
\label{0.vfcMain(a)}
[\ov\M_{A,g,n, s}^J(X)]^{vir}\ \in\ \cHH_*(\ov\M_{A,g,n, s}^J(X); \Q).  
\eest 
\item[(b)] For each $J\in \J(X, V)$,  there is a class
\best
\label{0.vfcMain(b)}
[\ov\M_{A,g,n, s}^J(X, V)]^{vir}\ \in\ \cHH_*(\ov\M_{A,g,n, s}^J(X,V); \Q).  
\eest
\end{enumerate}
\end{theorem}

\bigskip

For the cases not covered by Theorem~\ref{VFCfinecase} --- those where the dimension of the target is low and the genus of the domain is high  ---  Axiom~A4 must be modified by the addition of correction terms.  This complication will be addressed elsewhere. The case $A=0$ is much easier and can be treated with standard techniques.

\bigskip

 The image of such a virtual fundamental class under an $se$ map is  called a {\em Gromov-Witten class}.

\begin{cor}
\label{1.GWCor}
The Gromov-Witten class 
$$
GW_{A,g,n,s}(X,V) = se_*[\ov\M_{A,g,n,s}^J(X,V)]^{vir} \in H_*(\oM_{g,n}\ti X^n \ti  \P_s (NV))
$$
depends only on the  deformation class of $(J,\w)\in \J(X,V)$.
\end{cor}

\begin{proof}
Two tame structures $J_0=(J, \w)$ and $J_1=(J',\w')$ are deformation equivalent  if they lie in  a path $P=\{J_t\}$ in $\J(X)$.  Let $\iota_0: \{J_0\}\to P$ and $\iota_1:\{J_1\}\to P$ be the inclusions.    Then by  Axiom A2
$$
se_*\iota_{0*}[\ov\M^{J_0}(X,V)]^{vir} = se_*\iota_{1*}[\ov\M^{J_1}(X,V)]^{vir}= se_*[\ov\M^{P}(X,V)]^{vir}.
$$
\end{proof}

\medskip

 Theorem~\ref{VFCfinecase}    is proved  by repeatedly applying three moves:  (i) introducing $G$-structures and  stabilizing divisors,  and passing to the associated moduli spaces,   (ii) enlarging the space of almost complex structures and  (iii)  applying the continuity property of  \v{C}ech homology. Stabilizing divisors are analogs of  the   very ample line bundles used to construct the virtual fundamental class in algebraic geometry.   We follow   Cieliebak and Mohnke's approach  \cite{cm} of  using divisors from Donaldson's Theorem, and also include ideas from Siebert \cite{siebert} and McDuff-Wehrheim \cite{mw}.

 \medskip
 
  \noindent{\bf Acknowledgements.}   This work was  supported in part by the NSF under grants DMS-0905738 (for the first author) and 
DMS-1011793 (for the second).   We thank M.~Tehrani and A.~Zinger for pointing an error in the first version of this paper which affected the main result in low dimensions. We also benefited from very useful conversations and correspondence with J.~Morgan, J.~Pardon, C.~Wendl and C.~Woodward.

\vspace{1cm}

\setcounter{equation}{0}
\section{Families  of curves} 
\label{section2}
\medskip

Gromov-Witten moduli spaces $\oM(X)$ are families of pseudo-holomorphic maps.   To understand them, one must first understand the case where $X$ is a single point.  In this case,  $\oM(X)$  is a family of curves with an appropriate topology.    Thus it is not enough  to consider curves in isolation; one must consistently work with families of curves and maps between such families.  This viewpoint, reviewed in this section, is standard  in algebraic geometry, but its power has not been fully used in the geometric analysis literature on Gromov-Witten theory. 

  For our purposes, a {\em family of  genus $g$ $n$-marked curves over $S$} is a  proper surjective map of complex analytic spaces with sections $\sigma_1, \dots, \sigma_n$ 
\bear
\label{1.1}
\xymatrix{
\C\ar[d]_{\pi} \\
S \ar@/_1pc/[u]_{\si_i} 
 }
\eear
 whose fibers $C_s=\pi^{-1}(s)$  are closed, connected curves of  arithmetic genus $g$ and that has the following structure.  There is an nodal set $\N\subset \C$ such that (i)  $\pi$ is  a locally trivial fibration in  the complement of each neighborhood of  $\N$ and (ii)  for each point  $c\in \N$ there are coordinates $(x,y,v)\in \cx^2\ti V$ on a neighborhood of $c=(0,0,0)$ and  $(z,v)\in \cx\ti V$ on a neighborhood of $\pi(c)$  in which $\pi$ is given by  $\pi(x,y, v)=(xy, v)$.   Finally, the images of the sections $\sigma_i$ are disjoint and disjoint from $\N$.

Maps between families are given by the obvious commutative square; this gives a notion of  isomorphisms and automorphisms of families.  The pullback of a  family (\ref{1.1})  along  a map $\phi:T\to S$ is  a family $\phi^*\C$ over $T$.  One can then  envision a  {\em universal curve}:   a family $\ov\U \to \ov\M$ over some space $\ov\M$ such that  any  family (\ref{1.1}) is isomorphic to the pullback along a unique map  $\phi:S \to \ov\M$.
\bear
\label{1.2}
\xymatrix{
\C \ar[d]\ar[r]& \ov\U\ar[d] \\
S \ar[r]^\phi  & \ov\M
 }
\eear

In this generality,  universal curves  do not exist. The key impediment is the existence of families that contain a  fiber $C_s$  with non-trivial automorphisms.   If the map $\phi$ to $\ov\M$ is to be unique, then the pullback must be  $C_s/\Aut(C_s)$ instead of $C_s$.   Worse,  if  $\Aut(C_s)$ contains a subgroup isomorphic to $\cx^*$ then $\ov\M$ cannot exist as a Hausdorff space.   Thus the usual approach is to:
\begin{enumerate}
\item[(i)]  Restrict to {\em stable families}:  a  curve $C$ is  {\em stable} if  $\Aut(C)$ is finite and  a family is stable if its fibers are all stable.
\item[(ii)]  Require only the  weaker universal property that  there is a unique  map $\phi$ as in (\ref{1.2}) so that the pullback family  is a finite quotient of $\C\to S$.
\end{enumerate}
 Property (ii) makes $\ov\M$ a ``coarse moduli space'' and necessitates working with orbifolds.    (Alternative approaches using Artin stacks are not well-suited for use with geometric analysis.)

 \begin{ex}{\bf Deligne-Mumford spaces.}  Consider genus $g$ curves $C$ with $n$ distinct marked points $x_1, \dots, x_n$, none of which is a node.  The $x_i$, together with the nodes, are called the special points of $C$.  Smooth curves of this sort  are stable if $2g-2+n>0$;  in general $C$ is stable if each irreducible component with genus $0$ has at least 3 special points, and each with genus 1 has at least one special point.    For $2g-2+n>0$ families of stable curves modulo automorphisms are classified by maps into the Deligne-Mumford space $\ov\M_{g,n}$ (see \cite{lo2} for a nice survey of this).  Deligne, Mumford and Knudsen proved (see \cite{lo2}):
\begin{itemize}
\item $\oM_{g,n}$ is a projective, complex analytic orbifold;
\item  the locus $\partial\ov\M_{g,n} = \ov\M_{g,n}\setminus \M_{g,n}$ parametrizing singular curves is a normal crossing divisor in the orbifold sense.
   \end{itemize}
  \end{ex} 
There is also a   universal curve 
$p:\ov\U_{g,n}\ra\ov\M_{g, n}$ which is isomorphic to  $\ov\U_{g,n}= \ov\M_{g, n+1}$ where $p$ is the map that forgets the extra marked point $x_0$ adn collapses unstable components.   Each stable 
genus $g$ curve $C$ with $n$ marked points has a unique stable model $\st(C)=[C]\in \ov\M_{g, n}$ and a unique  map $\iota: C\ra \ov\U_{g, n}$ that fits in the commutative diagram  
\bear
\label{iota.gr}
\xymatrix{
C\ar[r]^{\iota\quad }\ar[d]\ar[dr]_{\st}&\;\ov\U_{g, n} \ar[d]^p \\
pt. \ar[r] & \ov\M_{g, n}
 }
\eear
($\iota$ is defined by $\iota(z)=\st (C, z) \in  \ov\M_{g,n+1}$ for non-special points  $z$ and extended by unique continuation).  However,  the fiber over  a stable curve $[C]\in \ov\M_{g, n}$ is isomorphic to $C/\Aut\, C$ and the map $\iota$ is not injective if $\Aut\, C\not= 1$; this means that 
$\ov\M_{g,n}$ is only a ``coarse moduli space''.

\begin{rem}\label{R.ext.DM.unstable}
Diagram (\ref{iota.gr})  also applies to prestable curves in the stable range $2g-2+n>0$ and we further extend it into the unstable range by {\em formally} defining $\ov\M_{g, n}$ to be the topological space  $\ov\M_{g, 3-2g}$ obtained by adding the minimum number of marked points needed to stabilize.   Thus  we define $\oM_{0,0}$,   $\oM_{0,1}$ and $\oM_{0,2}$ all to be  $\oM_{0,3}$, which is a single point, and define $\oM_{1,0}$ to be  $\oM_{1,1}$.

 In the unstable range, these spaces must be used cautiously:  for them,  the fiber of the universal curve (\ref{iota.gr}) is a point and  the map $\iota$  is only defined modulo $\Aut\, C$, which is a Lie group of positive dimension. 
\end{rem}

One can also consider families of  curves with additional structure.  By defining  a decorated curve to be a curve $C$  together with a map  $\rho: T\ra C$ we can include  extra marked points,  a principal $G$-bundle,  a divisor or log structure, etc.   Morphisms and automorphisms of decorated curves are defined by the obvious commutative square.  Again, a decorated curve is called stable if its group of automorphisms is finite.

In the examples we will use,   decorations  have the form  $\rho:\tC \to C$ where $\tC$ is another curve.  For each, there is a moduli space  $\ov\M^{dec}_{g,n}$ that classifies isomorphism classes of stable  pairs $(C, \rho)$ where $C$ is a genus $g$ nodal curve with $n$ marked points.    These decorated moduli spaces come with a universal curve $p:\ov\U^{dec}_{g, n}\ra \ov\M^{dec}_{g, n}$  and forgetful maps with the following properties:
\begin{enumerate}[(a)]
\item $ \ov\M^{dec}_{g, n}$ and $  \ov\U^{dec}_{g, n}$ are complex projective orbifolds. 
\item For each stable decorated curve $(C,\rho)$  there is a classifying map $\iota$ and a stabilization map $\st$ such that  $st=p\circ \iota$ whose image is a fiber isomorphic to $(C, \rho)/\Aut(C, \rho)$:
\bear
\label{iota.gr.2}
\xymatrix{
(C, \rho)\ar[r]^{\iota\quad }\ar[dr]_{\st}&\;\ov\U^{dec}_{g, n} \ar[d]^p  \\
& \ov\M^{dec}_{g, n}
 }
\eear
\item the maps (\ref{iota.gr.2}) are natural with respect to  ``forget decorations'' maps $\wt t$ and $t$  and ``forget marked point'' maps $p$:
\bear\label{U.M.graph.dec}
\xymatrix{
\ov\U_{g, n}^{dec} \ar[d]_{p}  \ar[r]^{\wt t}&    \ov\U_{g, n}     \ar[d]^{p}
&&  \ov\M_{g, n+1}^{dec} \ar[d]_{p}  \ar[r]^{\wt t}&     \ov\M_{g, n+1}   \ar[d]^{p}
\\ 
\ov\M_{g, n}^{dec}  \ar[r]^{t}& \ov\M_{g, n}   && \ov\M_{g, n}^{dec}  \ar[r]^{t}& \ov\M_{g, n}
}
\eear
\end{enumerate}

\begin{ex} 
Given a curve $C$ with $n$ marked points $\{x_i\}$, the operation of ``adding $\ell$ additional  marked points''  can be viewed as the choice of a degree 1  holomorphic map $\rho:\tC\to C$ where $\tC$ has $n+\ell$ marked points $\{\tilde{x}_i\}$ with $\rho(\tilde{x}_i)=x_i$ for $i=1, \dots, n$.  Then $\ov\M_{g, n}^{dec} $ is $\ov \M_{g, n+\ell}$. 
\end{ex}

\begin{ex}   Similarly,  $\rho$ could be a configuration of  $\ell$ unordered, non-special points on $C$ (i.e. an effective divisor on $C$); then $\ov\M_{g, n}^{dec}$ is the $\ell$-fold configuration space 
$\mathcal{C}^\ell_{g,n}\ra \ov\M_{g, n}$. We could also consider the relative Hilbert scheme 
$\mathcal{H}ilb^\ell\ra \ov\M_{g, n}$. These two spaces are homeomorphic, but have a different smooth (orbifold) structure; the second is a smooth resolution of the first. (e.g. in the case $C=\P^1$, $\mathcal{H}ilb^\ell(\P^1)=\P^d \ra \Sym^d\P^1$).  
\end{ex}

In these examples,  $(C, \rho)$ is stable whenever $\tC$ is, even if $C$ is not stable.  Thus  decorations can be used to stabilize curves, as was done  in Remark~\ref{R.ext.DM.unstable}.

Given that automorphisms cause problems, and  that decorated curves have fewer automorphisms, one should ask:  is there a choice of decorations that completely eliminates automorphisms? More specifically, is there a category, whose objects are families of decorated curves, in which no object has non-trivial automorphisms?  The resulting moduli space would then be a  {\em fine}  moduli space:    every family would be the pullback of a  classifying map, unique up to isomorphism,  that is an isomorphism on each fiber.   

 The existence of fine moduli spaces of curves was unknown until  the  2003 work of   Abramovich, Corti and Vistoli \cite{acv}.    They constructed  moduli spaces $\ov\M^{G}_{g, n}$ of stable  twisted $G$-covers for finite groups $G$.   These are $G$-covers $\rho:\wt C\ra C$  (possibly ramified, but only over special points of $C$) where $\wt C$ has marked points with stacky structure recording the ramification order of the cover, and that are also ``balanced'' (satisfy a matching condition at the nodes, cf. \cite[\S 2.2 and  \S4.3]{acv}).  They also showed that there exists certain finite groups $G$ such that  $\ov\M^{G}_{g, n}$  is a fine moduli space.  Their results are our final example. 
 
\begin{ex}\label{Ex.G} In genus zero $\ov\M_{0, n}$ is already a fine moduli space whenever  $n\ge 3$. 
For any finite group $G$ the moduli spaces $\ov\M^G_{g, n}$ (denoted ${\mathcal B}^{tei}_{g, n}(G)$ in \cite[Defn 7.6.2]{acv}) of stable twisted 
$G$ covers  have the properties (a), (b) and (c) above plus the following extra properties (cf. \cite{acv}): 

\begin{enumerate}
\item[(d)]   There is a natural $G$-action on  $\ov\M^G_{g ,n}$, and the forgetful map 
 $\phi_G:\ov\M^G_{g, n} \ra \ov\M_{g, n} $ is a Galois cover. 
\item[(e)] The universal curve $\ov\U_{g, n}^{G}$ is an open and closed  subset of $\ov\M_{g, n+1}^{G}$. 
\item[(f)] For each $g$ there exits a finite group $G$ such that  $\ov\M^G_{g ,n}$ and $\ov\U^G_{g ,n}$ are  smooth   complex   projective manifolds,  $\ov\M^G_{g ,n} $ is a fine moduli space and $\phi_G:\ov\M^G_{g ,n}\ra \oM_{g, n}$ is a finite surjective morphism (cf. \cite[Thm 7.6.4]{acv}).
\end{enumerate} 
Note that (f) implies the remarkable fact  that  {\em the  moduli space $\ov\M_{g, n}$ is a global quotient  orbifold} --- the quotient of a compact K\"{a}hler manifold by a finite group (this was first proved by Looijenga \cite{l}).    For later use, also note that the forgetful map $\phi_G$ induces a map in rational homology 
\best
\phi_{G*}: \mathrm{H}_*( \ov\M^G_{g,n}; \Q) \longra \mathrm{H}_*( \ov\M_{g, n}; \Q)
\eest
that relates the rational fundamental classes by the formula
\bear\label{deg.t}
\phi_{G*})[\ov\M^G_{g,n}]=\deg \phi_G\cdot  [\ov\M_{g,n}]. 
\eear
\end{ex}
\vspace{1cm}

\setcounter{equation}{0}
\section{Moduli spaces of holomorphic maps} 
\label{section3}
\medskip

Let $X$ be a closed symplectic manifold and let 
$
\J(X)
$
denote the space of   tame structures $(J,\om)$ on $X$,  consisting of an almost complex structure $J$ and a symplectic form $\w$, with the $C^0$ topology.   For each pair $(J,\w)$ the moduli space 
$$
\ov\M^J(X)
$$
 consists of equivalence classes of stable maps $f:C\ra X$ from a nodal marked domain $C$ into $X$  that are solutions to the $J$-holomorphic map equation
\best
\del _J f=0.
\eest
Such a  map is {\em stable} if the automorphism group $\Aut (f, C)$ is finite, and two such maps are equivalent if they are related by pre-composition by an  isomorphism of the domain. In particular, if  $C$ is a stable curve then any map $f:C\ra X$ is stable. In fact, the stability condition on $f:C\ra X$ is topological: there are no unstable domain components that represent 0 in homology. The moduli space has components  $\ov\M_{A, g, n}(X)$   indexed by the homology class $A=f_*[C]\in H_2(X, \Z)$,  the arithmetic genus $g$  and  number of marked points $n$ of 
$C$, and  is stratified by the topological type of the domain as explained in  the Appendix.

Each map has an  energy 
\best
E(f) =\frac 12 \int _C |df|^2 d {\rm vol}_C 
\eest
calculated using the Riemannian metric $g(u,v)=\frac12[\w(u, Jv)+\w(v, Ju)]$.  For $J$-holomorphic maps  the energy is the cohomological pairing  
\bear
\label{3.wEinequality}
E(f)=\om (f_*[C])=\om(A).
\eear
When compactness is important, we will restrict attention to the portion of the moduli space  {\em below energy level $E$}, written and defined as
$$
\ov\M^{J, E}(X)\ =\ \bigcup_{A, g, n} \ \Big\{f\in\ov\M^J_{A,g,n}(X)\, \Big|\, \mbox{$E(f)\le E$ and $3g-3+n\le E$}\Big\}.
$$
for $(J,\w)$ in $\J(X)$.  
Occasionally we may need to use a separate upper bound $E_1$ for $E(f)$ than the bound $E_2$ for $3g-3+n$, in which case $E=(E_1, E_2)$.\\

As $(J,\w)$ varies in $\J(X)$, the moduli spaces $\ov\M^J(X)$ fit together in an universal family $\ov\M(X)$ over $\J(X)$.  This  comes with several natural maps: the projection $\pr$ whose fiber at a fixed structure $J$ is the moduli space $\ov\M^J(X)$ and, for each $(A,g,n)$,  a stabilization-evaluation map $se$
\bear\label{3.M.diag}
\xymatrix{
 \ov\M_{A, g, n} (X) \ar[d]^{\pr} \ar[r]^{\qquad se \qquad} &  \ov \M_{g, n} \ti X^n 
\\ 
 \J(X)  &
}
\eear
 whose first component  takes the equivalence class of $f:C\ra X$ to the stabilization $st(C)$ in the Deligne-Mumford space $\ov\M_{g,n}$ and whose second component evaluates $f$   at the $n$ marked points (for  $2g-2+n\le0$ we define $st(C)$ as in   Remark \ref{R.ext.DM.unstable}).

To incorporate  maps with decorated domains, let  $\ov\M^{dec}_{A, g, n} (X)$ be the space of equivalence classes $[f, C, \rho, J]$ where $f:C\to X$ is  $J$-holomorphic map in $\oM_{A,g,n}(X)$ and $\rho$ is a decoration on $C$. The equivalence relation is given by diagrams
\bear
\label{3.decequivalence}
\xymatrix{
\tC'  \ar[d]^{\rho'}  \ar[r]_\cong &\tC\ar[d]^{\rho} & \\
 C' \ar[r]^\cong & C\ar[r]^f & X
}
\eear
As before, $[f, C, \rho, J]$ is stable if $\Aut (f, C, \rho)$ is finite or, equivalently, if the restriction of $(f,C,\rho)$ to every 
irreducible component   of $C$  either has finite automorphism group  or represents   a nontrivial homology class in $X$.

Decorated moduli space come with  a forgetful map $\phi :\ov\M^{dec}_{A, g, n} (X) \ra \ov\M_{A, g, n} (X) $ that forgets the decorations and which is natural with respect to the diagram  (\ref{3.M.diag.dec}),  so there are maps
\bear\label{3.M.diag.dec}
\xymatrix{
  \ov\M_{A, g, n}^{dec}(X) \ar[d]^{\phi} \ar[r]^{\qquad se \qquad } &  \ov \M_{g, n}^{dec} \ti X^n
   \ar[d] \\ 
  \ov\M_{A, g, n}(X) \ar[r]^{\qquad se \qquad } &  \ov \M_{g, n} \ti X^n 
}
\eear
of moduli spaces over $\J(X)$.

\subsection{Graphs maps and Ruan-Tian perturbations}
\label{Section3.1}

Restrict now to maps decorated by a twisted $G$-cover of their domains as in Example \ref{Ex.G} ($G$ could be the trivial group), and let $\oM^G(X)\ra \J(X)$ denote their moduli space that fits in the diagram (\ref{3.M.diag.dec}).

  A map $f:C\ra X$  whose domain is stable with $\Aut\; (C)=1$  has a graph 
$$
F:C\ra \ov\U_{g, n}\ti X
$$
defined by $F(z)=(\iota(z), f(z))$ where $\iota$ is the map in (\ref{iota.gr}) from $C$ to a fiber of the universal curve.  More generally, if  $f:C\ra X$ is a map whose domain is a stable twisted $G$-cover $\rho:\wt C\ra C$  (possibly with $G=1$) then the ``graph'' of $f$ is defined using  Diagram (\ref{iota.gr.2})  by 
\bear\label{graph.constr}
F=\iota_\rho\ti (f\circ \rho): \wt C\ra \ov\U^G_{g,n}\ti X.
\eear 

This also extends to the case when $\wt C\ra C$ is not necessarily stable as in Remark \ref{R.ext.DM.unstable}, in which case $F$ factors through the stable model $\st(\wt C\ra C)$. 
Observe that if $\Aut (C, \rho)=1$ then $\iota_\rho$, and therefore $F$, is an embedding.

Next  recall that the universal curve $\ov\U^G_{g,n}$   is projective; denote it by $\ov\U$ and fix an embedding $\ov\U\hookrightarrow \P^M $.    
 For each fixed $J$, let ${\mathcal V}(X)$ be the space sections of the bundle $\Hom^{0,1}(\pi_1^* T\P^M,  \pi_2^* TX)$ over $\P^M\ti X$  (i.e such that $J\circ \nu+\nu\circ j=0$).  Each $\nu\in{\mathcal V}(X)$ defines a deformation $J_\nu$ of the product almost complex structure on $ \ov\U \ti X$ by writing 
\bear\label{3.J.nu}
J_\nu= j\ti J \oplus (-\nu\circ j)= \l(\begin{matrix} j&0\\ -\nu\circ j &J\end{matrix}\r) 
\eear 
\begin{defn} 
\label{def3.2}
Let  $\JV(X)$ denote the space of  smooth  almost complex structures $J_\nu$ in the form  (\ref{3.J.nu});  its elements  can be written as  
pairs  $(J, \nu)$ with $J\in \J(X)$ and $\nu$ as above. 
\end{defn}
The equation $\del_{J_\nu} F=0$ on the graph $F:C\ra \ov\U \ti X$ or equivalently 
\bear\label{del=nu}
\del_J f(z)=\nu(z, f(z)) 
\eear
on the map $f:C\ra X$ will be called a  {\em Ruan-Tian perturbation} of the $J$-holomorphic map equation.   In the case of $G$-decorated maps, the corresponding equation on the pair 
$(f, \rho)$ is 
\bear\label{del=nu.dec}
\del_J \tilde{f} (z)=\nu(\iota_\rho(z),\; \tilde{f}(z)) 
\qquad\mbox{where $\tilde{f}=f\circ \rho$}.
\eear
Note that the map $(J, \nu)\mapsto J$ gives a fibration with a section $J\mapsto(J,0)$:
\bear\label{forget.nu}
\xymatrix{
\JV(X)\ar[r]&\ar@ /_1pc/ [l]   \J(X)
}
\eear

There is a corresponding extension of the universal moduli space $\oM^G(X)$ over $\JV(X)$ 
with a diagram (\ref{3.M.diag.dec})  where $\J(X)$ is replaced by $\JV(X)$. Note that $\JV(X)$, and therefore the perturbed pseudo-holomorphic map equation depends on the type of decorations used.

 The graph construction allows us to extend the space of allowable deformations (perturbations) of the pseudo-holomorphic map equation in a {\em very specific way to domain dependent ones}, defined by pullback from $\ov\U\ti X\hookrightarrow \P^M \ti X$ via the graph construction. As we will later see, the graph construction, combined with a fixed embedding $\ov\U \hookrightarrow \P^M$ has several other nice analytical consequences.

\vspace{1cm}

\setcounter{equation}{0}
\section{The topology of  $\ov\M_{g,n}(X)$}
\label{section4}
\medskip

The topology of the universal moduli spaces $\ov\M_{g,n}(X)$  is constrained by the requirement that  the maps used in Gromov-Witten theory  be continuous.  These constraints lead directly to a metric space topology on moduli space with all the desired properties --- see  Theorem~\ref{3.TopTheorem}.  Moreover, this topology is the same as the Gromov topology commonly used in the literature (see the Appendix). Everything in this section applies to the moduli spaces $\ov\M^G_{g,n}(X)$, but for simplicity we will often omit the $G$ from the notation.  

To topologize $\J(X)$, fix a Riemannian metric on $X$ and use the Sobolev   $W^{\ell,p}$ norms with  $\ell p\geq \dim X$.

\begin{lemma}
\label{4.Jtopology}
 The space $\J^\ell$      of tame $W^{\ell,p}$  pairs $(J,\w)$ on $X$ and the space  $\JV^\ell$ of triples $(J, \nu, \w)$ with $\nu$ as in Definition~\ref{def3.2} are smooth separable Banach manifolds.   
 \end{lemma}

  \begin{proof}
Let $V$ be the Banach space  all $W^{\ell,p}$ $2$-forms that satisfy $d\w=0$ weakly.
The equation $J^2=-Id.$ defines a fiber bundle $F$ that is a submanifold of $End(TX)$. Let ${\cal F}^\ell$ be  the completion of the space of smooth sections of  $F$ in the $W^{\ell, p}$ norm.  Because of the  Sobolev embedding $W^{\ell, p}\subset C^0$ for $\ell p\geq \dim X$,  ${\cal F}^\ell$ (or the completion in 
any norm stronger than $C^0$) is a smooth Banach  manifold \cite{pal}.  Then $\J^\ell$  is an open set (determined by the tame condition and the non-degeneracy of $\w$) of the  manifold  ${\cal F}^\ell \ti V$ and hence is a manifold.
\end{proof}

To topologize the (universal) moduli space $\ov\M_{A,g,n }(X)\to \JV^\ell$ recall that  $\ov\M_{A,g,n }(X)$ is the set of isomorphism classes of triples $(f,C,\wt J)$ where $C$ is a (not necessarily stable)  genus $g$ curve with $n$ marked points, $\wt J$ denotes a triple $(J,\nu, \w)\in  \JV^\ell$, and  $f:C\to X$ is a $(J, \nu)$-holomorphic map.   This set comes with several natural  maps:  (i) a projection $\pi$ to the space of parameters $\J=\JV^\ell$, (ii)  stabilization-evaluation maps  maps $se$, (iii) the energy function $E(f)$,  (iv) forgetful maps $\phi_k$ and $\phi_G$ that forget the last $k$ marked points or the decorations $G$ and contract those components that become unstable.  All of these maps are invariant under reparametrization and  therefore descend to the moduli spaces, giving commutative diagrams:  
\bear
\label{3.naturalmaps}
\xymatrix{
\oM_{A, g,n}^G(X) \ar[r]^{se} \ar[d]_\pi & \oM_{A, g,n}^G\times X^n\\
\J & 
}
\hspace{.5cm}
\xymatrix{
\oM_{A, g,n+k}^G(X) \ar[r]^{se} \ar[d]^{\phi_k} & \oM_{g,n+k}^G\times X^{n+k} \ar[d]^{\phi_k}\\
\oM_{A, g,n}^G(X) \ar[r]^{se} & \oM_{g,n}^G\times X^{n}\\
}
\hspace{1cm}
\xymatrix{
\oM^G_{A, g,n}(X) \ar[d]_{\phi_G}\\
\oM_{A, g,n}(X) & 
}
\eear
These maps respect  the action of the symmetric group $S_n$  permuting the marked points, and when the domains are twisted $G$-curves also respect the $G$ action.

The maps in (\ref{3.naturalmaps}) induce another  important collection of maps.  First,  each $f\in\ov\M_{A,g,n}^G(X)$ is continuous,  so its ``graph''  (cf. (\ref{graph.constr}))
\bear
\label{3.defGammaf}
\Gamma_f =Image (F)= se_0(\phi_1^{-1}(f))\subset \ov\M_{g, n+1}^G\ti X
\eear
is a  compact subset of $\ov\M_{g, n+1}^G\ti X$.   Thus $f\mapsto \Gamma_f$ defines a map
\bear
\label{3.graphmap}
\Gamma:\oM_{A,g,n}^G(X) \ra Subsets_{c}(\ov\M_{g, n+1}^G\ti X)
\eear
where $Subsets_{c}(Z)$ denotes the set of compact subsets of $Z$.
 Similarly,   for  each $k\geq 2$ there is a ``multi-point graph map''
$$
\Gamma^k:\oM_{A,g,n}^G(X) \ra Subsets_{c}(\ov\M_{g, n+k}^G\ti X^k)
$$
defined by $\Gamma^k_f =se(\phi_k^{-1}(f))$ using (\ref{3.naturalmaps}).

 \begin{defn}
\label{3.Very-rough}
A topology  on the collection of all moduli spaces
$$
\oM_{g,n}^G(X)
$$
is called {\em natural} if  (i) the energy function and   all the maps $\pi$, $se$,  $\phi_G$, $\phi_k$ and $\Gamma^k$ above are continuous, and (ii) the actions of $G$ and $S_n$   are continuous.
\end{defn}
The following theorem gives the key topological properties of Gromov-Witten moduli spaces, extending Theorem 5.6.6 of \cite{ms2}.   Recall that a map is {\em perfect} if it is continuous,  surjective, closed and  all  fibers  are compact.
 
 \begin{theorem}
 \label{3.TopTheorem}
There is a natural topology on the universal moduli spaces $\ov\M_{g,n}^G(X)\to\JV^\ell(X)$ that is  metrizable and for which
\begin{enumerate}[(a)]
\item Each map $\pi: \oM^{G, E}_{g,n}(X)\ra \JV^\ell(X)$ is proper. 
\item The  maps  $\phi_k$ in  (\ref{3.naturalmaps}) are
 proper and perfect.  
\end{enumerate} 
\end{theorem}

\begin{proof}
The proof is given in the appendix.
\end{proof}

Theorem~\ref{3.TopTheorem} shows that  the universal moduli space is Hausdorff.  Statement (a) is a version of the Gromov Compactness Theorem;  it implies that the fiber $\oM^J_{A,g,n}(X)$  over each $J\in\J(X)$ is compact.     The Appendix also contains a proof of the fact (Corollary~\ref{A.cor}) that the topology in Theorem~\ref{3.TopTheorem} is, in fact, the usual Gromov topology defined in the literature.

 \vspace{1cm}


\setcounter{equation}{0}
\section{The VFC for domain-fine moduli spaces}
\label{section5}
\medskip
 
This section shows how standard results imply the existence of virtual fundamental classes for one very nice class of  moduli spaces: the 
``domain-fine''  moduli spaces defined below.  These are especially easy to work with because  the stabilization map takes their domain isomorphically to a fiber of the universal Deligne-Mumford curve, and because  gluing theorems apply without complications.   

\medskip

The following two terms will be used repeatedly in this and later sections.  
\begin{defn}
\label{defStableMap}
A stable map $f:C\to X$  
\begin{itemize}
\item  is {\em domain-stable} (ds) if  $C$ is a stable curve, and
\item is  {\em domain-fine} (df) if in addition,   $\Aut\; C=1$.

\end{itemize}
 A moduli space $\ov \M(X)$ is called  domain-stable  (resp. domain-fine) over $U\subset\J^\ell$ if, for every $J\in U$, each map
in $\ov\M^J(X)$ is  domain-stable  (resp. domain-fine).  
\end{defn}
Both properties are preserved under adding decorations:   if $\ov\M_{A, g, n}(X;J)$ is domain-fine or domain-stable, then  so is $\ov\M_{A, g, n+k}^{G}(X; J)$ for any $k\ge 0$ and any  finite group $G$.

 \medskip

\noindent{\bf Examples.}\vspace{-2mm}
\begin{description}
\item [1] For the trivial class $A=0$,  the moduli space $\ov \M_{0, g, n}(X)$ is domain-stable over all of  $ \JV^\ell$.
\item [2]  For genus $g=0$, all domain-stable maps are  domain-fine. 
\item[3] If $X$ is a curve of genus $g\ge 2$,  all stable maps  $f:C\to X$  are domain-stable (those of positive degree have  $\mbox{genus}\, C\geq 2$, and degree 0  stable maps  have stable domains).
\end{description}

\medskip

\begin{lemma}
\label{JstableOpenLemma}
For  fixed topological data  $(A, g, n)$, the sets
\best
\label{J.st.A}
\J_{ds}^\ell =\left\{ J \, |\, \oM^{J}_{A, g, n}(X)\mbox{ is domain-stable} \right\}  
\hspace{15mm}
\J_{df}^\ell = \left\{ J \, |\, \oM^{J}_{A, g, n}(X)\mbox{ is domain-fine} \right\}  
\eest
are open in the $C^0$ topology on $\J^\ell$.  The spaces $\JV_{ds}^\ell$  and $\JV_{df}^\ell$,   defined similarly,  are also open.
\end{lemma}

\begin{proof}
 Under Gromov convergence,   the order of the automorphism group is upper semi-continuous and limits of  unstable domain components are unstable.  Thus each $dr$ (resp. $df$) map $f$ has a neighborhood with the same property.  For $J\in \J_{ds}^\ell$  (or $\J_{df}^\ell$) these open sets cover  the moduli space  $\oM^{J}_{A, g, n}(X)$,  and hence by compactness cover the moduli spaces $\pi^{-1}(U)$ for an open neighborhood $U$ of $J$.  The same argument applies to $\JV_{ds}^\ell$  and $\JV_{df}^\ell$.
\end{proof}

 The next theorem describes the structure of moduli spaces of domain-fine  perturbed $J$-holomorphic maps for  fixed  $(A, g, n)$.  These facts are well-known; we  sketch the proof and refer the reader to \cite{rt1} and \cite{rt2} for details.   As in those papers,  a parameter $J\in \J^\ell$ is called  {\em regular} for 
$(A, g, n)$ if  the linearization of the $J$-holomorphic map equation on each stratum with fixed topological type is a surjective operator,  for every $J$-holomorphic map $f\in \ov\M^{J}_{A,g,n}(X)$. 
  
\begin{theorem} 
\label{Ruan-Tian}
 If $J\in\JV^{\ell}$ is both domain-fine and regular for $(A, g, n)$ then    
 $\M^{J}_{A,g,n}(X)$  is an oriented  manifold  of dimension
 \bear\label{dim.MX}
\iota= 2c_1(X)A +( \dim_\R\, X-6)(1-g)+2n,
\eear
with a compactification $ \ov\M^J= \M^J \cup B^J$ whose  ``boundary'' $B^J$  is a finite union of strata, each a manifold of dimension $\le \iota-2$.

Furthermore, over a regular path  $\gamma$ in $\JV^{\ell}_{df}$ the moduli space  $\M^{\gamma}_{A,g,n}(X)$ is an  oriented   cobordism of dimension $\iota+1$   with a compactification $\oM^{\gamma}=\M^{\gamma} \cup B^\gamma$ with codimension 2 boundary $B^\gamma$.  Finally, the set of regular $J$ is open and dense  in $\JV^{\ell}_{df}$, and the set of regular paths is open and dense in the space of paths in $\JV^{\ell}_{df}$. 
\end{theorem}
\begin{proof} For domain-fine maps,  one can use the variation in $\nu$ to show that the linearization of the equation (\ref{del=nu}) in $(J, \nu)$ is onto (essentially because the graph $F$ of $f$ is an embedding, thus somewhere injective).  Standard results then imply that  there is a dense set of regular values in $\JV_{df}(X)$ over which each stratum of $\ov\M_{A,g,n}^{J}(X)$ is smooth with dimension equal to the index of the linearization; the top stratum  $\M^{J}_{A,g,n}(X)$, consisting of maps with smooth domains, has dimension \eqref{dim.MX},  while the stratum consisting of maps whose domains have $k$ nodes has dimension $2k$ less than \eqref{dim.MX}.  The top stratum is  oriented by the determinant line bundle as in \cite{rt1} or \cite{ms2}.   The compactness statements follow from Theorem~\ref{3.TopTheorem}a. 
 \end{proof}

\medskip

The properties of the moduli space listed in Theorem~\ref{Ruan-Tian} imply that, for regular domain-fine $J$, the compactified moduli space carries a  fundamental cycle 
 in rational  \v{C}ech homology  (as stated, the theorem does not provide enough information to obtain a fundamental class in singular homology).   The key point is that, in \v{C}ech homology, every oriented topological manifold $\M^J$ carries a fundamental class $[\M^J]\in \cHH_\iota(\M^J, \Q)$ (compactness is not needed),  and when $B^J$ has codimension~2   as  in Theorem~\ref{Ruan-Tian},       the long exact sequence for the  pair $(\ov{\M}^J, B^J)$ gives an isomorphism
 $$
 \cHH_\iota(\M^J, \Q)\cong  \cHH_\iota(\ov\M^J, \Q).
 $$
 Under this isomorphism,  $[\M^J]$ corresponds to a class in $\cHH_\iota(\ov\M^J, \Q)$  that we denote $[\ov{\M}^J]$.  The existence of this class was noted by Donaldson and Kronheimer \cite[page 357]{DK}, and a detailed presentation will be given in   \cite{IPFC}.

\medskip

These facts about \Cech homology can be used to define a virtual  fundamental class  for the moduli spaces 
$\ov\M^J_{A, g, n}(X)$
that are the fibers of
\bear
\label{5.MtoJ}
\pi: \ov\M_{A, g, n}(X)  \longrightarrow \JV_{df}^\ell
\eear
over the domain-fine $\JV_{df}^\ell$ as follows. Using Theorem~\ref{3.TopTheorem},   the moduli space (\ref{5.MtoJ}) is a space with two metric topologies:    one constructed as in the proof of Lemma~\ref{3.TopTheorem} with $\J$ replaced by $ \JV_{df}^\ell$,  and a ``rough topology''  obtained by the same construction but  using  the $C^0$ distance  on  $\JV_{df}^\ell$ in formula \eqref{3.Hdist}.  Fix  $J_0\in\JV_{df}^\ell$.   By Lemma~\ref{JstableOpenLemma} there is a $C^0$ ball $U\subset \JV^\ell$ containing $J_0$ that lies in $\JV_{df}^\ell$. Theorem~\ref{Ruan-Tian} then shows that:

\begin{itemize}
\item For a dense set of regular $(J,\nu)$ in $U$ the moduli space over $(J,\nu)$ has a fundamental class
\bear
\label{5.FirstFundamentalClass}
[ \ov\M_{A, g, n}^{J,\nu}(X)]\in \cHH_\iota( \ov\M_{A, g, n}^{J,\nu}(X), \Z).
\eear 

\item For a dense set of regular paths $\gamma$ in $U$ with endpoints $(J_0, \nu_0)$ and $(J_1, \nu_1)$, the images of the maps induced by the inclusions  into the moduli space $\ov{\M}^{\gamma}$ over $\gamma$ are equal:
\bear
\label{5.1.FirstFundamentalClass}
[ \ov\M_{A, g, n}^{J_0,\nu_0}(X)]=[ \ov\M_{A, g, n}^{J_1,\nu_1}(X)] \in \cHH_\iota( \ov\M_{A, g, n}^{\gamma}(X), \Z).
\eear
\end{itemize}

\medskip

 \Cech homology has a second property that is important for our purposes:  it satisfies the continuity axiom: 
\bear\label{cech.cont} 
\ma\varprojlim_i \cHH_*(Y_i; \Q) =\cHH_*(Y; \Q)
\eear
for any  inverse system $\{\cdots \to Y_3\to Y_2\to Y_1\}$ of compact metric spaces with limit $Y$ (cf. \cite{milnor}).   For computations, it is useful to note that for locally compact Hausdorff spaces, rational \Cech homology
coincides with the Steenrod homology theory described in   \cite{milnor} and in Section~4 of \cite{ma2}).  Taking $\{Y_i\}$ to be a sequence of regular moduli spaces leds to our first theorem about virtual fundamental classes.

\begin{theorem}
\label{5.CechcycleLemma} Fix $(A, g, n)$ and $\JV_{df}$ as in  Lemma~\ref{JstableOpenLemma}.  The fundamental class (\ref{5.FirstFundamentalClass}) extends uniquely to a  \v{C}ech homology class  
\bear\label{VFC.J}
[ \ov\M_{A, g, n}^{J}(X)]^{vir} \in \cHH_*( \ov\M_{A,g,n}^{J}(X), \Q)
\eear
defined for every domain fine $J\in \JV_{df}$  with the property that for any smooth path $\gamma$ in $\JV_{df}$ from $J_0$ to $J_1$ the images under the maps induced by the inclusion 
\bear\label{VFC.ga}
[ \ov\M_{A, g, n}^{J_0}(X)]^{vir}=[ \ov\M_{A, g, n}^{J_0}(X)]^{vir}  \in \cHH_*( \ov\M_{A,g,n}^{\gamma}(X), \Q)
\eear
are equal. In particular, the GW class 
\bear
\label{5.firstGW}
GW_{A, g,n}(X)\;\ma =^{\mathrm {def}}\;(\st\ti \ev )_*[ \ov\M_{A, g, n}^{J}(X)]^{virt}\in H_*(\ov \M_{A,g,n} \ti X^n, \Q)
\eear
 is independent of $J$ on each path-component of  $\J_{df}$.  
\end{theorem}
\pf
For simplicity we write $J$ to mean a pair $(J,\nu)$. For each $J\in \JV_{df}$, consider the balls $B_k \subset\JV^\ell$ consisting of all   $J'$ whose distance from $J$ in the $W^{\ell, p}$ norm is less than  $1/k$.   By Lemma~\ref{JstableOpenLemma},  $B_k$ lies in  $\JV_{df}$ for large $k$.   Furthermore, each  $B_k$ is path connected, contains a dense set of regular values 
$J$ for which (\ref{5.FirstFundamentalClass}) holds, and any two regular values  are connected by a regular path 
for which (\ref{5.1.FirstFundamentalClass}) holds. 

\begin{wrapfigure}[8]{r}{0.35\textwidth}
\psfrag{a}{$J_1$}
\psfrag{b}{$J_2$}
\psfrag{c}{$J_3$}
\psfrag{f}{$J'_2$}
\psfrag{g}{$J'_3$}
\psfrag{h}{$J$}
\psfrag{e}{$J_1'$}
\centering
      \includegraphics[width=0.3\textwidth]{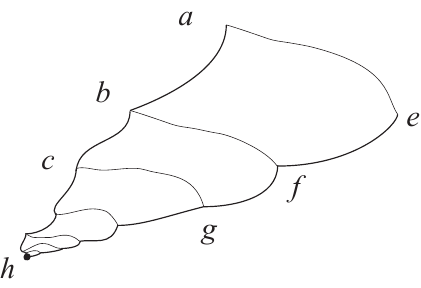}  
\end{wrapfigure}  
Choose a sequence $J_k\in B_k$ of regular points converging in $C^0$ to $J$ and regular paths $\gamma_k\subset B_k$ from $J_k$ to $J_{k+1}$.  For each $m$ 
$$
K_m \ =\ \{J\}\  \cup \  \bigcup_{k\ge m}\gamma_k
$$
 is compact,  and  $\ov\M_m=\pi^{-1}(K_m)=\ov\M^{K_m}_{A, g, n}(X)$ 
is  a  sequence of nested compact metric spaces (cf. Theorem~\ref{3.TopTheorem}) whose intersection is the compact  space $\ov\M^{J}_{A, g, n}(X)$.

For each regular value $J_k\in K_m$ the  images under the inclusions $\ov\M^{J_k}(X)\hookrightarrow \ov\M_m$  determine  a  rational  \v{C}ech homology class
\bear
\label{5.Cechdata}
VFC_{m}\;\ma=^\mathrm{def}\; [ \ov\M_{A, g, n}^{J_k}(X)]\in \cHH_\iota(\ov\M_m,  \Q).
\eear
As in   (\ref{5.FirstFundamentalClass})   and (\ref{5.1.FirstFundamentalClass}), this class is independent of $k$, and these homology classes are consistently related by the  inclusions $\ov\M_{m_1}\hookrightarrow \ov{\M}_{m_2}$ for $m_1\ge m_2$. The continuity axiom of rational   \v{C}ech homology then shows that there is  a well-defined limit 
\best
[ \ov\M^J(X)]^{vir}=\varprojlim_m VFC_m  \in\varprojlim_m \cHH_*(\ov{\M}_m, \Q) = \cHH_*( \ov\M^J_{A, g, n}(X), \Q)
\eest
 at each point $J\in \JV_{df}$.   If $K_m'$ is another such broken path, we can find regular paths  between $J_k$ and $J'_k$ inside $B_k$ as shown in the figure.  Then (\ref{5.1.FirstFundamentalClass}) shows that the two classes $VFC_m$ are compatible in the homology of the moduli space over $K_m\cup K'_m\cup\Gamma_m$.
 
 Relation (\ref{VFC.ga}) follows the same way by joining broken paths with different limit points in the same path component of $\J_{df}$.
Finally, observe that the rational classes (\ref{5.firstGW}) are locally constant in $J$, so are constant on path connected components of $\J_{df}$.  
\qed

 \vspace{1cm}
 

\setcounter{equation}{0}
\section{The VFC for domain-stable moduli spaces}
\label{section6}
\medskip

The construction of Section~5 defines a virtual fundamental class  over those elements of  $\JV^\ell$ with domain-fine moduli spaces.  We  next extend the construction to the larger class of domain-stable moduli spaces by decorating  domains with twisted $G$-covers.  

\medskip
 
As described after (\ref{3.M.diag}),   an element of $\ov\M^G(X)$ is an isomorphism class of data $(f, C, \rho, J)$ consisting of a map $f:C\to X$ and a  
$G$-twisted curve $\rho:\tC\to C$.   The $G$-structures    define a lifted  graph map (\ref{graph.constr}) into  $\ov\U^G\ti X\subset \P^M\ti X$ and an enlarged   space of perturbations  $\JV^G(X)$;  these perturbations depend on both $f$ and $\rho$.    Over $\J_{ds}(X)$ (but not over $\JV^G_{ds}(X)$!) the map $(f, \rho)\mapsto f$ that forgets the twisted $G$-cover defines     a  map $t$ that appeared in  (\ref{3.naturalmaps}) and that fits into the diagram
\bear
\label{6.diagram1}
\xymatrix{
\oM^G(X) \ar[d]_t\ar[r]& \oM^G(X)\ar[dd] \\
\ov\M(X) \ar[d]  &  \\
\J(X) \ar[r] & \JV^G(X)
 }
\eear
Equivalently, $t$ is the quotient map for the action of  $G$ on $\oM^G(X)$ by $g[f,\rho]=[f, \rho g^{-1}]$. There is also a  $G$-action on $\JV^G(X)$  induced from the action on universal curve $\ov\U^G$, but the righthand vertical map in (\ref{6.diagram1}) is  not $G$-invariant. In this context,  the results of Section~\ref{section5} hold on the set   $\JV_{df}^G(X)$ of $J\in \JV^G(X)$  for which the moduli space $\ov\M^{G,J}_{A,g,n}(X)$ is domain-fine:

 \begin{lemma}
\label{6.Ruan-Tian2}  Fix   $(A, g,n)$ and a finite group $G$. Then  Theorems~\ref{Ruan-Tian}  and \ref{5.CechcycleLemma} hold for the moduli spaces $\ov\M^G_{A,g, n}(X)$ over $\JV_{df}^G(X)$.  Hence there is a well-defined virtual fundamental class
\bear
\label{6.VFCondf}
[\M^{G, J}_{A,g, n}(X)]^{vir}
\eear
\end{lemma}
 \begin{proof} As described at the end of Section~3, the pair $(f, \rho)$ is $(J, \nu)$-holomorphic if and only if its graph $F$ is $J_\nu$-holomorphic. Because of  the domain-fine assumption,      $\Aut(C, \rho)=1$ and hence  $F$ is embedding into the manifold $\P^M\ti X$.   McDuff and Salamon show (Chapter 3 of \cite{ms2}) that  generic perturbations of $J_\nu$ on $\ov\U^G\ti X$ give regularity for somewhere injective maps.  From their proof, one sees that it is enough to use perturbations of the form (\ref{3.J.nu}).  As in 
 Theorem~\ref{Ruan-Tian}, regular domain-fine moduli spaces carry a virtual fundamental class.  The proof of Theorem~\ref{5.CechcycleLemma} then applies without change.

Note that,   in proving  regularity,  the relevant linearization to the perturbed equation is obtained by varying $f\circ \rho$ through {\em $G$-invariant} maps.  This variation is determined by the variation in   $f$ because  the set of equivalence classes of  twisted $G$-covers $\rho:\tC\to C$ over a fixed $C$ is finite. (The linearization of the $J$-holomorphic map equation for  the map $\tilde{f}=f\circ \rho$ is  a different operator with a different  index). 
 \end{proof}

  We next verify  that the classes  (\ref{6.VFCondf}) are consistently defined when one replaces $G$ by a larger group.  Given an extension $0\to G\to K \to H\to 0$ of $H$ by $G$, there is an action of $G$ on the space of $K$-twisted maps whose quotient induces a map 
 \bear
 \label{6.HKmap}
 t:\ov\M_{A, g, n}^{K}(X) \ra \ov\M_{A, g, n}^{H}(X)
 \eear
 defined over $\J(X)$ and more generally over $\JV^H(X)$ where $\JV^H(X)\hookrightarrow\JV^K(X)$ is the map $\nu\ra t^* \nu$  induced by the map $t:\ov\U^K\ra \ov\U^H$ at the level of universal curves. 
 \begin{lemma} 
 \label{6.lemma2} Assume $J\in \J(X)$ is domain-fine for both  $\ov\M_{A, g, n}^{H}(X)$ and 
 $\ov\M_{A, g, n}^{K}(X)$. Then the virtual fundamental classes of (\ref{6.HKmap}) are defined and related by 
\bear\label{6.VFC.G}
 \frac{1}{|H|} [ \ov\M_{A, g, n}^{H, J}(X)]^{vir} \ =\ \frac{1}{|K|}\  t_* [ \ov\M_{A, g, n}^{K,J}(X)]^{vir}\in 
 \cHH_*( \ov\M_{A, g, n}^{H, J}(X), \Q)
 \eear
 \end{lemma}
\begin{proof} Since the domain-fine is an open condition, there exists a sufficiently small $C^0$ neighborhood $U$ of $J$ in $\JV^H(X)$ over which both moduli spaces remain domain-fine.  By Lemma~\ref{6.Ruan-Tian2} there is an open dense set of regular $J'=(J, \nu)$ in $U\subset \JV_{df}^H(X)$ and for these the moduli space 
 $\ov\M_{A, g, n}^{H,J'}(X)$  is a compact oriented topological manifold. For the pullback $J'$ to $\JV^K(X)$ the moduli space $\ov\M^K(X)$ is by assumption also domain-fine, so the $G$ action on it is free and therefore the linearized operator for the space of $K$-twisted  maps  (described in the proof of Lemma~\ref{6.Ruan-Tian2}) is also regular.  

 Thus $\ov \M_{A, g, n}^{K,J'}(X)$ is also a compact oriented topological manifold and the quotient by the $G$-action gives the map
 $$
 t: \ov\M_{A, g, n}^{K, J'}(X)  \to  \ov\M_{A, g, n}^{H,J'}(X) 
 $$
 which has virtual degree $|G|=|K|/|H|$.  Thus (\ref{6.VFC.G}) holds for an open dense set of 
 $J'\in \JV^H_{df}(X)\cap \JV^K_{df}(X)$, and hence for all $J\in \J_{df}^H(X)\cap \J_{df}^K(X)$  by the continuity construction of Theorem~\ref{5.CechcycleLemma}.
\end{proof}

 \medskip

 Next consider  a moduli space $\ov\M_{A, g, n}^{J}(X)$ which is only domain stable.  Then 
 $\ov \M^{H, J}_{A, g, n}(X)$ is  domain stable for any finite group $H$. Let $G$ be any finite group with property (f) of Example \ref{Ex.G}, and consider the moduli space $\oM^{H\ti G}(X)$ of twisted $H\ti G$-structures, which is then domain-fine, and  therefore Lemma \ref{6.Ruan-Tian2} applies to it.  Referring to diagram (\ref{6.HKmap}), define 
  \bear\label{6.VFC.GH}
 [ \ov\M_{A, g, n}^{H, J}(X)]^{vir} \ =\ \frac{1}{|G|}\  t_* [ \ov\M_{A, g, n}^{H\ti G,J}(X)]^{vir}\in \cHH_*( \ov\M_{A, g, n}^{H, J}(X), \Q).
 \eear

  \begin{theorem} 
  \label{6.maintheorem} 
  The moduli spaces $ \ov\M_{A, g, n}^H(X)$ over the space of domain-stable $\J_{ds}(X)$ admit a virtual fundamental class (\ref{6.VFC.GH})  that   
  \begin{enumerate}
\item[(i)] is independent of the $G$ used in its definition,
\item[(ii)] is consistent under the quotient map as in (\ref{6.VFC.G}), 
\item[(iii)] for $H=1$ it extends the one over the domain-fine $\J_{df}(X)$ defined by Theorem~\ref{5.CechcycleLemma}.
\end{enumerate}

 \end{theorem}
\begin{proof}
Define the virtual fundamental class by (\ref{6.VFC.GH}).  If $G_1$, $G_2$ are two  finite groups that  satisfy condition (f) of  Example \ref{Ex.G}, then so does $G_1\ti G_2$.  Applying Lemma~\ref{6.lemma2} to   the diagram  

$\oM^{H\ti G_1\ti G_2}(X) \ra  \ov\M^{H\ti G_i}(X)\ra \oM^H(X)$  shows that the virtual fundamental class (\ref{6.VFC.GH}) induced by $G_1$ and $G_2$ are both equal to  the one induced by $G_1\ti G_2$ so (i) holds. Part (ii) follows in a similar fashion from Lemma~\ref{6.lemma2}.  Part (iii) follows by taking $H=1$ in Lemma~\ref{6.lemma2}. \end{proof}

A little more work extends the virtual fundamental class of Theorem~\ref{6.maintheorem} over the space $\JV^G(X)$ in Diagram~\ref{6.diagram1}.   For each $(J,\nu)\in \JV^G(X)$,  $G$ acts on the moduli space $ \ov\M^{G, J, [\nu]}(X)$ over the orbit $[\nu]$.  For a dense set of generic $(J,\nu)$, the action of $G$ on $(J,\nu)$ will be free and the moduli space $ \ov\M^{G, J, [\nu]}(X)/G $ will be regular. Applying the limiting argument of Theorem~\ref{5.CechcycleLemma} then extends the VFC over $\JV_{ds}^G/G$ and defines  
 \best
 [ \ov\M_{A, g, n}^J(X)]^{vir} \ = \ \frac 1{|G|}\lim_{\nu\ra 0} [ \ov\M_{A, g, n}^{G, J, [\nu] } (X)/G]\in 
 \cHH_*( \ov\M_{A, g, n}^J(X), \Q).
  \eest
  for all $J\in \J_{ds}(X)$.

  \vspace{1cm}
 

\setcounter{equation}{0}
\section{Relative moduli spaces}
\label{section7}
\medskip

The  set of moduli spaces $\ov\M_{A,g,n}^G(X)$ is contained in the larger set of  moduli spaces relative a divisor $V\subset X$. For $G=1$, this was  done in the \cite{IP1} for smooth divisors, and was recently generalized in \cite{i-nor} to normal crossing divisors.   This section  reviews  those aspects of the theory that are needed later.

 A smooth divisor $V$ is an embedded codimension 2 submanifold of $X$ that is $J$-holomorphic for some tame pair $(J, \om)$. More generally, a normal crossing divisor $V$  in $X$ is the union of closed immersed codimension 2 submanifolds that are $J$-holomorphic for some  tame pair $(J, \w)$,   and are in general position; see Definition 1.3 in \cite{i-nor} for precise details.   For each pair $(X,V)$ there is a stratification of $X$ whose depth $k\ge 0$ strata $V^k$ correspond to points in $X$ where at least $k$ different local branches of $V$ meet. Moreover, for each integer $m\ge 0$ there is an associated deformation space ${\cal Z}_m\to D^m$ over the polydisk whose total space ${\cal Z}_m$ is a subset of a finite dimensional manifold.   For smooth divisors $V$, the generic fiber of ${\cal Z}_m$ is diffeomorphic to $X$ and the central fiber is $X_m=X\cup \P_V\cup \cdots \cup \P_V$,  obtained by attaching  $m$ copies of the compactified normal bundle   $\P_V=\P(\cx\oplus NV)$, identifying the zero section of one to the infinity section of the next.    For  general normal crossing divisors $V$ there is a similar deformation space whose central fiber $X_m$ is a level $m$ ``building'' constructed from the strata of $V$.   In all cases, there is also a map $\pi_V:X_m\to X$ that collapses all copies of $\P_V$ to $V$.

\medskip

The relative moduli spaces are constructed in much the same way as the usual moduli spaces $\oM_{A,g,n}(X)$ with three  main differences  (see \cite{IP1}  and \cite{i-nor} for details): 
\begin{enumerate}
\item[(a)]  One restricts $(J, \w)$ to  be in the subspace $\J^\ell(X,V)\subset\J^\ell(X)$ of  triples that satisfy   a condition (``$V$-adapted'')  on the 1-jet of $J$ along $V$. 
\end{enumerate}
In particular, this implies that $V$ is $J$-holomorphic, and thus  some $J$-holomorphic maps  into $X$ can have components mapped into  $V$.
\begin{enumerate}
\item[(b)]  For each map  $f\in \ov\M_{A,g,n}(X)$ with no components or nodes mapped to $V$, one marks all the points 
in $f^{-1}(V)$ and  records their  intersection multiplicities $s$ with the various branches of $V$ to obtain an open set $\M_{A, g, n, s}(X,V)$ of maps with extra $\ell(s)$ marked points and contact data $s$.
\item[(c)] For each $s$, there is a compactification $\ov\M_{A, g, n, s} (X, V)$ as described in \cite{i-nor}. 
\end{enumerate}
As a set, $\ov\M_{A, g, n, s} (X, V)$ consists of equivalence classes of certain type of maps $f:C\ra X_m$ where maps are equivalent if they are related by an  isomorphism of their domains and the  rescaling of their target  $X_m$ induced by the $\cx^*$ action on each positive level $\P_V$.  While all these $f$ are stable as maps into $X_m$, their projections $\wh f=\pi_X\circ f$ to $X$  may have unstable components that are mapped to a point of $V$.  Each such component, called a {\em trivial component}, has its image in a fiber of $\P_V$ and its domain is an unstable rational curve with precisely two marked points.
When $V$ is smooth, the maps $f$ in the compactification  have no nontrivial components in the total divisor of $X_m$, thus all nontrivial components of $f$ have well-defined contact data to the total  divisor and $f$ satisfies a matching condition along the singular divisor of $X_m$ (cf. \cite{IP1}). When $V$ has normal crossings, trivial components of $f$  in the singular divisor cannot be avoided, but then the nontrivial components satisfy an enhanced matching condition (cf. \cite{i-nor}).  This matching can be expressed in terms of the inverse image of a certain diagonal $\De$ under an  evaluation map $\Ev$ which fits in the diagram 
\bear\label{rel.pr.st.ev}
\pi\ti se\ti \Ev: \ov\M_{A, g, n, s}(X) \ra \J(X,V) \ti \left(\ov\M_{g, n+\ell(s)} \ti X^n\right) \ti N_sV
\eear
and which is a lift of the usual evaluation map.   (As described in \cite{i-nor}, $\Ev$  keeps track of the  multiplicities $s$ as well as of  the leading coefficients of  $f$ at the extra $\ell(s)$ contact points to $V$.)   Such a map $f$ is called {\em relatively stable} if its automorphism group is finite, or equivalently if $f$ has at least one nontrivial domain component in each positive level. 

\medskip

The relative moduli spaces come with several functorially-defined maps. There is a diagram
\bear\label{7.pi.to.X}
\xymatrix{
 \ov\M_{A, g, n,s}(X, V)\ar[r]^{\phi_V} \ar[d]_{\pi} & \ov\M_{A, g, n+\ell(s)} (X)\ar[d]_{\pi}
\\
\J(X,V) \ar@{^{(}->}[r] & \J(X)
}
\eear
where $\phi_V$ is the map that takes  $f$ to the map ${\wh f}$ obtained from $\pi_X\circ f$ after collapsing all the trivial domain components of its domain:
 \bear\label{contr}
\xymatrix{
C \ar[r]^f \ar[d]^{\ct}&X_m\ar[d]^{\pi_X}
\\
\wh C \ar[r]^{\wh f} & X
}
\eear
 Here $\wh C=\ct(C)$ is called the {\em contracted  domain} of $f$. It is different from the stable model $\st(C)$,  which is obtained by  contracting all unstable domain components, not just the trivial ones.

  The normal crossing divisors in $X$  form a directed system,  partially ordered by inclusion;  the  smallest element is the empty divisor.  There is a corresponding directed system of moduli spaces.  Whenever $V\subset V\cup V'$ there is a diagram generalizing  (\ref{7.pi.to.X})  
 \bear
\label{7.Vdirectedsystem}
\xymatrix{
 \ov\M_{A, g, n, s\cup s'}(X,V\cup V')\ar[d]_{\pi}\ar[rr]^{\phi_{V'}}& &
\ov\M_{A, g, n+\ell', s}(X,V) \ar[d]^{\pi}
\\ 
\J(X,V\cup V')\ar@{^(->}[rr] &&
\J(X,V)
}
\eear
that forgets some components and that is  compatible with the map (\ref{rel.pr.st.ev}).

 In the special case of Diagram~\ref{7.pi.to.X} when $(X, V)=(\P_V, V_0)$,  the projection $\pi:\P_V\ra
V$ induces a map 
 \bear\label{pi.to.V}
\xymatrix{
 \ov\M_{A, g, n,s}(\P_V, V)\ar[r]^{\phi_V} \ar[d]_{\pi} & \ov\M_{\pi_*A, g,
n+\ell(s)} (V)\ar[d]_{\pi}
\\
\J(\P_V, V) \ar[r] & \J(V)
}
\eear
except when $\pi_*A=0$, $g=0$ and $n+\ell(s)<3$.  In this context, we enlarge the definition of ``trivial component'' of $f$ to include unstable genus zero curves whose image under $\wh f: \wh C\ra V$ is a constant.  
\begin{rem}\label{R.ext.M.unstable} As discussed in Remark \ref{R.ext.DM.unstable}, it is convenient to extend the definition of the moduli space $\oM_{A, g, n}(X)$ into the unstable range 
$A=0$, $2g-2+n\le 0$ by defining $\ov\M_{0, g, n}(X)$ to be the topological space  $X\ti \ov\M_{g, n}$. With this definition, the map (\ref{pi.to.V}) extends to the unstable range.  \end{rem}

Most of the discussion in the previous sections extends to the relative case.  The next several paragraphs describe the minor modifications needed.

\bigskip

\noindent{\bf A. The topology of moduli space.}  The refined Gromov convergence described in \cite{i-nor} similarly induces a metrizable topology on the relative moduli space. This topology is the same as  the one  obtained by regarding the relative moduli space as a subset of  $\ov\M_{A, g, n+\ell(s)}({\cal Z})$ (with the images of all maps landing in the fibers of ${\cal Z}$).  
A sequence $f_n:C_n\ra X$ of relatively stable maps in $\M_s(X, V)$ converge if there exists an $m$ and a sequence of rescaling parameters $\la_n\in (\cx^*)^m$ of the target such that the rescaled maps 
$R_{\la_n} f_n$, regarded as maps from $C_n$ into the fiber of ${\cal Z}$ over $\la_n$ converge to a relatively map $f_0:C_0\ra X_m$ into the central fiber of $\cal Z$ over $\la=0$.  This now involves a choice of local trivialization/identification of both the domains and targets of the maps away from their singular locus.  

 With the topology of Lemma~\ref{4.Jtopology}, the bottom map in  (\ref{7.pi.to.X}) is a continuous injection.  We first give the relative moduli space $ \ov\M_{A, g, n,s}(X, V)$ the topology induced by pullback by the map $\phi_V$ in  (\ref{7.pi.to.X}) from the Gromov topology on the absolute moduli space $\ov\M(X)$. The graph maps  into $\ov\U\ti {\cal Z}$ similarly induce the metric topology on $\ov\M(X,V)$. 
 
\bigskip

\noindent{\bf B. Domain-stable and Domain-fine.}  Definition~\ref{defStableMap}  extends  to  relatively stable maps as follows.

 \begin{defn}\label{D.super-st} A  map $f\in \ov \M_{A, g, n,s}^J(X,V)$ is  {\em domain-stable} (resp. {\em domain-fine}) if  ${\wh f}$ is domain-stable (resp. domain-fine).
 \end{defn}
With this definition a moduli space $\ov\M_{A, g, n,s}(X,V)$ is domain-fine if and only if its image in $\ov\M_{A,g, n+\ell(s)}(X)$ under the map (\ref{7.pi.to.X}) is domain-fine. 

\bigskip

\noindent{\bf C. Twisted $G$-structures.} In the decorated version of the relative theory, the moduli space $\ov\M^G_s(X, V)$  consists of equivalence classes $[f, \rho]$ where $\rho:\wt C\ra C$ is a twisted $G$-cover and $f:C\ra X_m$ is a smooth map in a level $m$ building as before. The equivalence relation is now up to isomorphisms of the twisted $G$-covers $\rho: \wt C\ra C$ as well as the $(\cx^*)^m$ rescaling action on the target  $X_m$. The moduli space $\ov\M^G_s(X, V)$ comes with a natural $G$-action whose quotient $\ov\M^G_s(X, V)/G= \ov\M_s(X, V)$ is the original undecorated relative moduli space. 

\bigskip

\noindent{\bf D. Ruan-Tian perturbations.}
In Diagram~\ref{7.pi.to.X}, the image of the bottom map may not contain any regular $J$. Thus we will enlarge $\J^\ell(X,V)$ to  space $\JV^\ell(X,V)$ of perturbations using the construction of Section~\ref{Section3.1}.  Note, however, that this new space $\JV^\ell(X,V)$ does {\em not} embed in the space  
$\JV^\ell$ of Section~\ref{Section3.1}.   
 
Regarding relatively stable maps as maps into the fibers of ${\cal Z}$, there is a graph map $F:\tC\to \ov\U\ti {\cal Z}$ as in (\ref{graph.constr}) with $\ov\U= \ov\U_{g,n+\ell(s)}^G$. Again, we fix an embedding $\ov\U\ra \P^M$. Let $\JV^\ell(X,V)$ be the space of tame pairs $(J_\nu,\w)$ on  $\ov\U\ti {\cal Z}$ where $J_\nu$, generalizing (\ref{3.J.nu}), preserves the fibers of both the universal curve $\ov\U$ and of ${\cal Z}$, and also satisfies the $V$-compatibility condition (see Definition  3.2 of \cite{IP1} and its extension \cite{i-nor}).  Note that the graph map factors through ${\wh C}$, so trivial components remain trivial under these perturbations.   

Diagram (\ref{6.diagram1}) similarly extends to the relative case. 
\bigskip
 
\noindent{\bf E. Twisted decorated relative moduli spaces.} To describe a smooth model of the relative stable map compactification $\ov\M_s(X,V)$ for domain-fine maps,  we also need to include a  choice of roots of the leading coefficients  of $f$ at all the contact points with the total divisor (see \cite{i} and \cite{i-nor}).  Adding a choice of these roots defines another resolution $\ov \M^{s}(X, V) \ra \ov \M_s(X, V) $ of the relative moduli space and more generally $\ov \M^{G,s}(X, V) \ra \ov \M_s^G(X, V) $ of the decorated relative moduli spaces for any finite group $G$.
These fit as the top row in the diagram:
 \bear\label{MsG.to.Ms}
\xymatrix{
 \ov\M^{G, s}(X, V)\ar[r]^{t } \ar[d]_{\pi} & \ov\M^s(X, V) 
 \ar[d]_{\pi}
\\
\ov\M_s^G(X, V) \ar[r]^t &\ov\M_s(X, V)
}
\eear
The moduli spaces in the top row come  a natural action of $\Mu_s$, the product of the cyclotomic groups of roots of unity at each one of the contact points to $V$, whose quotient gives rise to the moduli spaces in the bottom row, while the moduli spaces in the left column come with a natural $G$ action whose quotient are the spaces in the right column. Of course, there is also a symmetric group action $S_\ell$ reordering the contact points to $V$. 

\begin{ex} For a fixed nodal marked curve $(C, x)$,  the relative moduli space $\ov\M_s(C, x)$ is the space of admissible covers constructed by Mumford-Harris, and while its resolution $\ov\M^s(C, x)$ is the moduli space of twisted (balanced) covers constructed by Abramovich-Vistoli.  Moreover, when 
$\Aut(C, x)=1$, the space of twisted $G$-covers of $(C, x)$ considered in \cite{acv} is a particular  example of the relative moduli space $\ov\M^s(C, x)$ of twisted covers $\rho:\wt C\ra C$.  \end{ex}

When $G=1$ and $V$ is smooth, the boundary strata of $\ov\M(X, V)$ consisting of maps into a level 1 building $X_1=X\cup \P_V$ can then be described in terms of its resolution 
\bear\label{resolution}
\ov\M^{s}(X, V)\ti_{\Ev} \ov\M^{s}(\P_V, V_\infty\cup V_0)\longra \ov\M(X, V)
\eear
obtained by (a) ordering the marked points corresponding to the nodes along the singular locus 
$V=V_\infty$, and (b) choosing a root of the leading coefficients of $f$ at these nodes. The local model of the relative moduli space near $f$ is 
\bear\label{la=mu}
\la = a_1b_1 \mu_1^{s_1}=\dots= a_\ell b_\ell \mu_\ell^{s_\ell} 
\eear 
where $a_i$, $b_i$ are the leading coefficients of $f$ at the node $x_i$. The zero locus of equation (\ref{la=mu}) is not smooth at the origin when $\ell\ge 2$. One obtains a smooth resolution of (\ref{la=mu}) by a base change after replacing the variables $(a_i, b_i)$ with $(\alpha_i, \beta_i)$ where $\alpha_i^{s_i}=a_i$ and  
$\beta_i^{s_i}=b_i$. The antidiagonal $\Mu_\ell$ action preserves the $\al_i \beta_i$ and gives rise to the balancing condition. There is also a symmetric group action $S_\ell$ that reorders the $\ell$ nodes. 

\medskip

 With this set-up, the dimension counts and transversality results proved in  \cite{IP1} and \cite{i-nor}  give the following analog of Theorem~\ref{Ruan-Tian}.

\begin{lemma}
 \label{7.Ruan-Tianlemma} 
For a Baire set of  $J\in \JV^G_{df}(X,V)$,  the relative moduli space 
$\M^J=\M_{A, g, n}^{G,s,  J}(X, V)$ is an oriented  manifold of dimension
 \bear
 \label{7.dim.MX}
\iota= 2c_1(X)A -2V\cdot A +( \dim_\R X-6)(1-g)+2(n+\ell(s))
\eear
with a compactification $ \ov\M^J= \M^J \cup B^J$ whose  ``boundary'' $B^J$  is a finite union of strata, each a manifold of dimension $\le \iota-2$. Furthermore, over a regular path $\gamma$ in $\JV^{\ell}_{df}$, the 
moduli space $\M^\gamma$  an  oriented   cobordism of dimension $\iota+1$ with a compactification  $\oM^{\gamma}=\M^{\gamma} \cup B^\gamma$ with codimension 2 boundary $B^\gamma$.   Finally, the set of regular $J$ is open and dense  in $\JV^{\ell}_{df}(X,V)$, while the set of regular paths is open and dense in the set of paths in $\JV^{\ell}_{df}(X,V)$. 
\end{lemma} 
As in Section~5, Lemma~\ref{7.Ruan-Tianlemma}  is all that is needed to obtain a virtual fundamental class in \v{C}ech homology; no  knowledge of how the strata fit together is needed.
 
\begin{theorem}
\label{7.Maintheorem}  The collection of moduli spaces in diagram (\ref{MsG.to.Ms}) admit compatible VFCs over  $\J_{ds}(X, V)$, and  for any path between $J_0$ and $J_1$ in $\J_{ds}(X, V)$ the VFC over $J_0$ and $J_1$ have the same image under the inclusion maps as in (\ref{VFC.J}). In particular 
\best
[ \ov\M_{A, g, n,s}^{J}(X,V)] ^{vir}=\frac 1{|\Mu_s||G|}\pi_*[ \ov\M_{A, g, n}^{G,s, J}(X,V)] ^{vir}  \in \cHH_*( \ov\M_{A, g, n,s}^{J}(X,V); \Q)
\eest
extends the one of   that appears in (\ref{6.VFC.GH}) when $V=\emptyset$. 
\end{theorem}

 \begin{rem}\label{R.symMS} The symmetric group $S_\ell$ acts on the extra $\ell=\ell(s)$ contact points to $V$ inducing an action on the moduli space $\ov\M_{s}(X, V)$ over $\J(X)$. In the arguments that follow, it will be convenient to also consider the quotient by this $S_\ell$ action, in which case the resulting relative moduli space
 \bear
\label{symmetrizedMS}
\ov\M_{[s]}^{J}(X, V)=\l. \ov\M_{s}^{J}(X, V)\r/S_{\ell(s)}
\eear
will be called the {\em symmetrized relative moduli space} and the subscript $s$ will change to $[s]$. It has a corresponding VFC
\best
[ \ov\M_{A, g, n,[s]}^J(X,V)] ^{vir}= \frac1 {\ell(s)!} \pi_*[ \ov\M_{A, g, n,[s]}^J(X,V)] ^{vir} \in \cHH_*( \ov\M_{A, g, n,[s]}^J(X,V); \Q)
\eest
defined as the pushforward by the quotient map $\pi: \ov\M_{s}(X, V)\ra \ov\M_{[s]}(X, V)$ over 
$\J_{ds}(X)$.  With this definition, the symmetrized relative $GW$ invariant corresponds to an unordered sequence of $\ell$ multiplicities $[s]$ and is equal to
\best
GW_{[s]}= \frac 1 {\ell!}  GW_s\in H_*(X^n\ti (\ov\M_{g,n+\ell}\ti V^\ell)/ S_\ell; \Q)
\eest

 \end{rem}

\medskip 

The upshot of this discussion is that the symmetrized relative moduli space $\ov\M_{[s]}(X,V)$ also has smooth representatives constructed by breaking the symmetry and turning on a generic (non-equivariant) perturbation. To fully break the symmetry, one needs to order the marked points and the components of the domain, add a suitable group $G$ for genus $g\geq 1$,  and finally also add the roots of the leading coefficients.

\vspace{1cm}


\setcounter{equation}{0}
\section{Stabilizing divisors} 
\label{section8}
\medskip

This  section  introduces the  notion of a stabilizing divisor  for a pair $(X, V)$ where $V$ is a normal crossing divisor in a closed symplectic manifold $X$.   Stabilizing divisors will be crucial in the next two sections.   Here we present some of their important properties and show that Donaldson's Theorem implies that stabilizing divisors exist in abundance on any symplectic manifold. 

 To start, we fix $X$,  a normal crossing divisor  $V$ and a subset $B$   of  $\J(X, V)$.   To enumerate the strata of moduli spaces,   let 
  $$
  \C_{E, B}
  $$ 
be   the collection of all triples $(A, g, k)$ with $A\in H_2(X)$ such that:
 \begin{enumerate}[(i)]
 \item $A\not= 0$.
\item $E(A,g)= \max \{\w(A), 3g-3\} \leq E$ (``energy less than $E$'').
\item For some $J\in B$,  $\oM^J_{A,g}(V^k)$ is not empty, i.e.  $A$ is represented by a $J$-holomorphic   map  from a genus $g$ curve  into the  stratum $V^k$ of  $(X, V)$ with depth $k\ge 0$ (including the top stratum $V^0=X$). 
\end{enumerate}

\begin{defn}
\label{def8.1}
A smooth codimension 2 submanifold $D\subset X$ is an  {\em $E$-stabilizing divisor for $(X, V, \omega)$ on $B$}  if   
 \begin{enumerate}[(i)]
\item $D$ is transverse to $V$ and there exists  $(\w, J')\in \J(X,V)$  in  $B$ such that $D$ is $J'$- holomorphic.
\item $D$ is sufficiently positive in the sense that  
\bear\label{st. ineq.V}
D \cdot A \ge c_1(\wt V^k)A+\dim_\cx V^k +2g+1 \qquad \mbox{for all }  (A, g, k)\in \C_{E, B}
\eear
where $\wt V^k$ is the smooth resolution of the  closed stratum $V^k$ of $V$ (cf. \cite{i-nor}). 
\end{enumerate}  
\end{defn}
\non If $D$ is a stabilizing divisor for $(X, V)$ we let $$\J_D(X,V)$$ be the set of all $(J,\w)\in \J(X,V)$ such that $D$ is $J$-holomorphic.

\begin{prop}
\label{M.s-stable.V} 
Suppose that $D$ is an  $E$-stabilizing divisor for $(X, V)$ at the point $B=(J_0, \w_0)$ in $\J(X,V)$. Then there is a $C^0$ ball $U$ around  $B$ in $\J_D(X, V)$ such that for an  open dense and  path-connected set  $\wh U$ of $J\in U$: 
\begin{enumerate}
\item[(a)] the only $J$-holomorphic genus $g$ maps into $D$ with $E(A,g)\le E$ are constant; 
\item[(b)] the relative moduli spaces $\ov\M_{A, g, n, s}^{J}(X,V\cup D)$ with $E(A, g)\le E$  are domain-stable; 
\end{enumerate}
Statements (a) and (b)  remain true in generic 1-parameter families $\{J_t\}$  in $U$. 
\end{prop}

\begin{proof}
This follows by a standard dimension count argument; note that (a) and (b) are $C^0$-open conditions on $J$. For simplicity, we first prove this for the case $V=\emptyset$, in which case (\ref{st. ineq.V})   becomes
\bear\label{st. ineq}
D \cdot A \ge c_1(X)A+\dim_\cx\, X+2g+1  \qquad \mbox{for all }  (A, g)\in \C_{E, J_0}
\eear
First notice that  (\ref{st. ineq}) is a topological condition on what kind of strata appear in the moduli space $\ov \M^{J_0}(X)$ below energy level $E$. By Lemma \ref{L.top.not.jumps}, there is a $C^0$-ball  $U$ around  $J_0$ with $\C_{E, U}=\C_{E, J_0}$,  so the inequality  (\ref{st. ineq}) holds  for all non-trivial   $J$-holomorphic maps with $J\in U$. 

Assume next that  $\dim\, D>0$  (the result is trivially true otherwise). Then below energy level $E$, the expected dimension of the moduli space $\M_{A, g, 0}(D)$ is 
\best
\dim_\cx \;\M_{A, g, 0}(D)&=&c_1(D)A+(\dim_\cx \, D-3)(1-g) 
\\
&=& c_1(X) A - DA+2g-3+ \dim_\cx\, D-(\dim_\cx\, D -1)g
\\
&\le& -\dim_\cx\, X+ \dim_\cx\, D-3-(\dim_\cx\, D -1)g\ <\ -1
\eest  
where for the first inequality we used  (\ref{st. ineq}) for any $J\in U$. The standard transversality argument at a somewhere injective map  \cite{ms2} implies that for generic $J$ in $U$  there are no simple smooth $J$-holomorphic  maps into $D$, and thus no multiple covers either, and this is true in a generic 1-dimensional family of $J$'s. Therefore the only $J$-holomorphic curves in $D$ below energy level $E$ are constants, and these have stable domains. Furthermore, in the relative moduli space 
$\M(\P_D, D_\infty\cup D_0)$ we see only multiple covers of the fiber (relative both $0$ and $\infty$) and these all have stable domain except for the trivial covers. Similarly, in  $\M(\P_D, D_\infty)$ we see only multiple covers of the fiber relative $\infty$ only, and these have stable domain unless $g=0$, $\ell(s)\ge 1$ and $n+\ell(s)\le 2$, i.e. are trivial components. 

 Next, by (\ref{st. ineq}) the expected dimension $d$ of the relative moduli space 
 $\M_{A, g,n, s}(X,D)$ is 
 \best
d &=&c_1(X) A + (\dim_\cx \,X-3)(1-g)+n+\ell(s)- D \cdot A
 \\
 &=&c_1(X) A -D\cdot A +2g-2+n+\ell(s) + (\dim_\cx \,X-1)(1-g)
 \\
 &\le &-\dim_\cx \,X +\ell(s)+n-3+(\dim_\cx\, X-1)(1-g)
 \eest
 If the domain is unstable, that is $g\le 1$ and $\ell(s)+n<3-2g$, this gives
 \best
\dim_\cx \;\M_{A, g, n,s}(X,D)<-1.
\eest 
Thus for generic $J\in U$ (even in a generic  path) there are no simple, smooth maps with unstable domain in $\M_s(X, D)$, and therefore there are no multiple covers  other than the constants (below energy level $E$).  Recalling the first part of the argument,  we conclude that for generic $J$ all maps in $\ov\M_s(X, D)$   have stable domains, and the same remains true in generic paths of $J$'s. This proves (b).

Next consider the general case when $V$ is a normal crossing divisor in $X$. Then the proof follows the same outline as above, except that now $V$ induces a stratification of $X$: the depth $k$ piece $V^k$ is where at least $k$ branches of $V$ meet (here $k\ge 0$ includes $V^0=X$); $V^k$ is smooth away from higher depth stratum and comes with a smooth resolution $\wt V^k \ra V^k$. Because we are now restricted to  
$J\in \J(X, V\cup D)$, to get transversality for the first part of the argument we need to work separately with simple maps into each stratum $V^k\cap D$ which are not contained in any higher depth stratum. Any such $J$-holomorphic map lifts to a map in the resolution $\wt V^k$, with image not contained in the higher depth stratum of $\wt V^k$. Furthermore, if $A$ denotes its homology class, then 
$c_1(\wt V^k\cap D)A=c_1(\wt V^k)A-D\cdot A$. The dimension count then shows that the only maps into $D$ below energy level $E$ are constants. For the second part of the argument, for each smooth, simple map $f\in\ov\M(X, V\cup D)$ consider its projection $\wh f$ to $\ov\M(X, V)$ under the map that collapses all the levels over $D$, but not those over $V$. There are two possibilities for this projection. If  $\wh f$ is a map into $D$, then by the first part of the argument it is constant; thus $f$ is a map into a fiber of $(\P_D, D_0\cup D_\infty)$; these all have stable domains except for the trivial maps. Otherwise 
$\wh f$ is a map into one of the strata $V^k$ that does not lie entirely in the higher depth stratum of 
$V\cup D$; then $\wh f$ has a lift to a map in a relative moduli space $\ov\M_s(\wt V^k, \wt {V^k\cap D})$, 
where $\wt {V^{k}\cap D}$ is the resolution of the lift of the divisor  $V^{k+1}\cap D$ to $\wt V^k$. 
The dimension of this moduli space is $c_1 (\wt V^k) A -D\cdot A+(\dim_\cx V^k-3)(1-g)+\ell(s)+n$ which is similarly negative by (\ref{st. ineq.V}) if the domain is unstable. 

Finally, note that (a) and (b)  in the statement of the Proposition are both open conditions, thus the subset 
$\wh U$  of $U$ on which they hold is open and dense set. Moreover, since $U$ is path connected and properties (a) and (b) hold for generic path in $U$  then $\wh U$ is also path connected. \end{proof}

\medskip

 The arguments  in the next section require a  controlled  existence statement for stabilizing divisors.   The following lemma shows that stabilizing divisors exist for a  dense set  of tame pairs $(J,\w)$ in the spaces $\J(X)$ and  $\J(X,V)$.   In the statement,   $(\om_0, J_0)$ is a compatible pair, but the general elements of the ball   $B_\ep$ are  {\em tame, but not necessarily compatible pairs}. 

\begin{lemma}
\label{exist.Don.div}
Fix an energy level $E$ and a compatible structure $(J_0, \om_0)$ on $X$  and let $B_\ep$  denote  $C^0$-ball around $(J_0,\w_0)$ as above.   Then there exists an $\ep_0>0$ such that, for any $0<\ep<\ep_0$, 
\begin{enumerate}[(a)]
\item  There exists an $E$-stabilizing divisor $D$ for $X$ on $B_\ep\subset \J(X)$.
\item If $V$ is a normal crossing divisor for $(J_0,  \om_0)$, there exists an $E$-stabilizing divisor $D$ for $(X, V)$ on  $B_\ep \cap \J(X,V)$.
\item   Given $(J_1, \w_0), (J_2, \w_0) \in B_\ep$ and   normal crossing divisors $V_1, V_2$ in $(X, J_i, \w_0)$, there is a path $D_t$ of  $E$-stabilizing divisors for $X$ on $B_\ep$ whose endpoints $D_i$ are $E$-stabilizing  divisors for $(X, V_i)$ on $B_\ep\cap \J(X, V_i)$. 
\end{enumerate} 
\end{lemma}
\begin{proof}
All three statements follow from Donaldson's Theorems \cite{d1} \cite{d2} and its extensions by Auroux \cite{a}; related statements were also proved in \cite{cm}. 
We include here an outline of the proof for the specific form stated above.

First, by Lemma \ref{L.top.not.jumps}  there exits an $\ep_0>0$ such that the $C^0$ closed ball $B_0$ centered at $(J_0,\om_0)$  and of radius $\ep_0$ satisfies  $\cal C_{B, E}=\cal C_{J_0, E}$.  Furthermore, there is a lower bound $\w_0(A)\ge 2\alpha_0>0$ for all $(A,g,k)\in\C_{J_0,E}$  Thus after shrinking $\ep_0$ we may assume that
\bear
\label{8.wAlowerbound}
\om(A)\ge \al_0 >0
\eear
for all $(\om, J)\in B_0$ and $(A, g,k)$ in the finite set $\cal C_{B, E}$.

Now fix a $C^0$ ball $B=B_\ep$ about the compatible pair $(\om_0, J_0)$ with $\ep<\ep_0$, sufficiently small so that all $J\in B$ are tamed by all the $\omega$'s in $B$ (being a tamed pair is an open condition).  Choose a compatible pair $(J', \w')$ in a ball $B'=B_{\ep/4}$ with $\om'$ representing a rational cohomology class.  Donaldson's Theorem implies that there exists  constants $m_\ep, C_\ep$ such that for any $m\ge m_\ep$, there exists a smooth divisor $D_m$ representing $m\om'$ and which is $C_\ep m^{-1/2}$-$J'$-holomorphic.  Hence for sufficiently large $m$:
\begin{itemize}
\item there exists an almost complex structure $J_m$ on $(X, D_m)$ such that $|J_m-J'|_{C^0}<\ep/4$ (see Appendix)  and therefore $J_m\in B\cap \J(X, D_m)$. 
\item    $D_m$ satisfies  (\ref{st. ineq}) on $B$ because $D_m\cdot A= m\w'(A)\ge m \alpha_0$    by (\ref{8.wAlowerbound}).
\end{itemize}
Therefore $D_m$ is an $E$-stabilizing divisor for $X$ on $B$, as statement (a) asserts.

Next, given a  normal crossing divisor $V$ in $(X,\om_0, J_0)$, similarly we can find a compatible pair $(J', \w')$ in  $B'=B_{\ep/2}$ such that  $\w'$ is  rational  and $J'$ is adapted to $V$ (so  $(\om', J')\in \J(X, V)$). For any $\eta>0$ small similarly by the results of Donaldson and Auroux,  there exists constants $C_\ep$, $m_\ep$ (which depend also on $\eta$) such that for each $m\ge m_\ep$ there is a smooth divisor $D_m$  with the following properties: (i) $D_m$ is Poincar\'{e} dual to  $m\om'$ (ii) $D_m$ is $  C_\ep m^{-1/2}$-$J'$ holomorphic  and (iii) $D_m$ is $\eta$-transverse to every stratum of $V$.  As above,  $D_m$ satisfies  (\ref{st. ineq}) on $B$ for all large $m$.   Then we can construct  a $J_m$ within $\ep/2$ in $C^0$ from $J_0$ (cf Appendix)  such that $J_m$ is $D_m\cup V$ adapted. Therefore $D_m$ is an $E$-stabilizing divisor for $(X, V)$ on  $B_\ep \cap \J(X,V)$. Note that the construction in the Appendix gives rise to a $J_m$ which is tamed by $\omega_0$  for $\ep$ small but perhaps not compatible. 

\medskip

The proof of (c) is similar.  For  $i=1, 2$ we can similarly  find $V_i$-adapted, compatible pairs $(\om', J'_{i})$ with the same rational $\om'$,  and  a path $(\w', J'_{t})$ of compatible  pairs in $B_{\ep/4}$ between them. Applying Donaldson's Theorem to the path $(\w', J'_{t})$ shows that, for each $\eta>0$ small enough, and each large $m$, there is  a family $D_{t, m}$ of smooth divisors with the following properties: (i) they are Poincare dual to $m\om'$ (ii) $D_{t, m}$ are $C_\ep m^{-1/2}$-$J_t'$-holomorphic and (iii) at the endpoints $D_{i, m}$ are $\eta$-transverse to $V_i$ for $i=1, 2$.  Again,   (\ref{8.wAlowerbound}) also implies that (iv) $D_{t, m}$ satisfy  (\ref{st. ineq}) on $B$ for all 
$0\le t\le 1$. Then the construction in the Appendix produces a perturbed path $J''_t$ of almost complex structures on $X$ still within $\ep/2$ of $J_0$ (thus tamed by $\om_0$) such that $J''_t$ is adapted to $D_{t,m}$ for all $t$ and moreover at the endpoints $J''_i$ is adapted to 
$V_i\cup D_{i, m}$. Therefore $D_{t, m}$ is a path of $E$-stabilizing divisors with the properties stated in (c). 
\end{proof}

\begin{rem} 
By first deforming $\om$ to a rational form and then using Auroux's generalization \cite{au2} of 
Donaldson's theorem,  we can similarly  find stabilizing divisors representing the Poincare dual of $\tau+k\om$ for any $\tau\in H^2(X, \Z)$ and $k>k_\tau$ sufficiently large. The case $\tau= c_1(TX)$ is useful in certain applications.
\end{rem}

\vspace{1cm}


\setcounter{equation}{0}
\section{The stabilization procedure} 
\label{section9}
\medskip

When $D$ is an $E$-stabilizing divisor for a normal crossing divisor $(X, V)$ it is useful to define a space of ``weakly compatible'' almost complex structures: we let $\J^E_D(X, V)$ denote the subspace of $\J(X,V)$ consisting of all $J$ such that
\bear
\label{defJVVD}
\hspace*{.3in}\begin{minipage}{5.5in}
\begin{enumerate}[(a)]
\item $D$ is $J$-complex, and
\item all $J$-holomorphic maps $f$ into $D$ with $E(f)\le E$ and $g\le E$ are constant.
\end{enumerate}
\end{minipage}
\eear
Here (b) weakens part of the definition of $V\cup D$-compatibility, the 1-jet condition on $J$ along $D$ in  Definition~3.2 of  \cite{IP1}. The 1-jet condition was  used only to prove that the linearized normal operator $D_f^N$ is complex linear (Lemma~3.3 of \cite{IP1}). But conditions (a) and (b) above  also imply complex linearity because $D_f^N$ is complex linear for constant maps.  Thus all results in \cite{IP1} hold for weakly compatible $J$. 

\medskip

For the remainder of this section,  fix  a normal crossing divisor $V$ in $X$ and a compatible pair $(\om_0,J_0)$ in $\J(X,V)$.  Then  fix topological data $(A,g)$ and an energy level $E$ with $E(A,g)<E$.  We will work under the hypothesis of Theorem~\ref{VFCfinecase}, namely
\begin{assumption}
\label{dim assumption}  Let $V$ be a normal crossing divisor on $X$. 
Assume that  $A\not= 0$ and that either $g\le 1$, or $X$ and every strata of the normal crossing divisor $V$ (if non-empty) has dimension at least 12.
\end{assumption}
\noindent Finally,  fix a $C^0$ ball  $B=B_\ep \subset  \J(X,V)$ centered at $J_0$ and small enough that Lemma~\ref{exist.Don.div} applies.  

\smallskip
Let $D$ be an $E$-stabilizing divisor for $(X,V)$ on $B$, which exists by Lemma~\ref{exist.Don.div}b.   The proof of Theorem \ref{7.Maintheorem}  combines with the discussion above to give the following:
\begin{lemma} 
Assume $D$ be an $E$-stabilizing divisor for $(X,V)$ on $B$, and $E(A, g)\le E$. Then for each $J\in  \J_D(X,V) \cap B$ there exists a virtual fundamental cycle
\bear\label{vfc.VD}
[ \ov\M_{A, g, n,s}^J(X,V\cup D)] ^{vir}  \in \cHH_*( \ov\M_{A, g, n,s}^J(X, V\cup D); \Q)
\eear
that is independent of $J$ in the sense that,  for any path $\gamma$ between  $J_0, J_1\in \J_D(X,V)\cap B$, 
the image under the inclusions are equal: 
\bear\label{vfc.VD.ind}
[ \ov\M_{A, g, n,s}^{J_0}(X,V\cup D)] ^{vir} = [ \ov\M_{A, g, n,s}^{J_1}(X,V\cup D)] ^{vir} 
 \in \cHH_*( \ov\M_{A, g, n,s}^\gamma(X, V\cup D); \Q)
\eear
\end{lemma} 
\begin{proof} Applying  Proposition~\ref{M.s-stable.V}  at each point of the non-empty path-connected set $ \J_D(X,V) \cap B$ and taking the union shows  that there is an open, dense path-connected   subset  $\hat{U}$ of  $ \J_D(X,V) \cap B$  so that 
\begin{enumerate}[(i)]
\item  $\hat{U}\subset  \J(X,V;  D)$ and 
\item  $\ov\M_{A,g,n, s}(X,V\cup D)$ is domain-stable for all $n, s$ and all  $J\in \hat{U}$.
\end{enumerate}
Theorem \ref{7.Maintheorem}  then gives a well-defined VFC 
on $\wh U$. This then extends by continuity to the entire $  \J_D(X,V) \cap B$ with the fact that $\wh U$ is open, dense and path connected subset of $\J_D(X, V)\cap B$, and that any path in $\J_D(X, V)\cap B$ can be approximated by a path in $\wh U$ (the set of paths in $\wh U$ are dense in the set of all paths in $\J_D(X, V)\cap B$). 
\end{proof}
Now consider the component of the moduli space  in (\ref{vfc.VD}) with $s'=s\cup [1]$, that is,  the compactification of the space of stable maps that intersect $V$ with multiplicity vector $s$ and intersect $D$ at $A\cdot D$ unordered points each with multiplicity 1.  As in Diagrams~(\ref{3.naturalmaps}) and (\ref{7.pi.to.X})  there is a projection
\bear
\label{9.projection}
\xymatrix{
\oM_{A, g,n, s'}(X, V\cup D) \ar[d]^{\phi_D}  \\
 \oM_{A, g,n,s}(X,V)
 }
\eear
where  $\phi_D$ is the composition of the map in diagram (\ref{7.Vdirectedsystem}) that forgets the branch $D$ with the map that also forgets the marking of the contact points to $D$.

\begin{defn} 
\label{9.defnVFC.1} Assume $(X, V)$ satisfies Assumption \ref{dim assumption}. 
If $D$ is any $E$-stabilizing divisor for $(X,V)$ on $B$, for any $J\in \J_D(X,V) \cap B$  the image of the homology class (\ref{vfc.VD}) under $\phi_D$ defines a class
\best
VFC^J_D\;\ma=^{def} \;\phi_{D*}[ \ov\M_{A, g, n,s\cup [1]}^J(X,V\cup D)] ^{vir}  \in \cHH_*( \ov\M_{A, g, n,s}^J(X, V); \Q)
\eest
\end{defn}
\non Here $[1]$ means that we  first divide by the symmetric group action reordering the contact points to $D$ as described in Remark \ref{R.symMS}.  

As defined $VFC^J_D$ depends on the choice of $D$ (as well as $V$ and ${A,g,s}$ which are assumed fixed in the discussion below).  The remainder of this section is devoted to showing that $VFC^J_D$ is independent of $D$ in an appropriate sense.

\medskip

The next step is to prove that these virtual fundamental classes are consistent for different choices of $D$.  We start with a proposition whose proof is based on   a separate, independent result proved in Section~11. 
 
\begin{prop}
\label{9.differentD}
If  $D$ and $D'$ are  $E$-stabilizing divisors for $V$ on $B$  then for any $J\in B\cap \J_{D\cup D'}(X,V)$  (if one exists)
\bear\label{VFC.D=VFC.D'}
VFC_D^J=VFC_{D'}^{J} \in \cHH_*(\oM^{J}_{A, g,n, s}(X, V), \Q).
\eear
\end{prop}
\begin{proof} By a slight modification of Proposition~8.2, there is an open dense and path connected subset $\wh U$ of $B\cap \J_{D\cup D'}(X,V)$ on which 
\begin{enumerate}[(a)]
\item  the moduli spaces 
$\ov\M^{J}_{A, g, n, s\cup [1]}(X, V\cup D)$, 
$\ov\M^{J}_{A, g, n, s\cup [1]}(X, V\cup D')$ and $\ov\M^{J}_{A, g, n, s\cup [1]}(X, V\cup D\cup D')$ are all  domain-stable and 
\item  there are no  $J$-holomorphic maps into $D$ or $D'$ other than constants with energy less or equal to $E(A, g)$. 
\end{enumerate}
Proposition~\ref{11.2}  below  shows that these conditions, together with Assumption~\ref{dim assumption}, imply that 
$$
VFC_D^J= VFC_{D\cup D'}^J= VFC_{D'}^{J}
$$
 for all $J\in \wh U$. This equality extends by continuity to all $J\in B\cap \J_{D\cup D'}(X,V)$.
\end{proof}
\begin{theorem}\label{cor.D=D'} Assume $D$ and $D'$ are $E$-stabilizing divisors for $(X, V)$ on $B$. Then for any $J\in B\cap \J_D(X, V)$ and $J'\in B\cap \J_{D'}(X, V)$ there exists a path 
$\gamma$ in $B\cap \J(X, V)$ from $J$ to $J'$ such that
\best
VFC_{D}^J= VFC_{D'}^{J'} \in  \cHH_*(\oM^{\gamma}_{A, g, n, s}(X, V), \Q)
\eest
after inclusion. 
\end{theorem} 
\begin{proof} For simplicity, set $\J=\J(X,V)$. Use Lemma 8.3 (c) to find a path $D_t$ of $E$-stabilizing divisors on $B$ for $(X, V)$ and whose endpoints are $E$-stabilizing divisors for $(X, V\cup D)$ and respectively $(X, V\cup D')$ on $B$ thus $B\cap \J_{D\cup D_0}\ne \emptyset$ and $B\cap \J_{D'\cup D_1}\ne \emptyset$.  Since $B\cap (\cup_t \J_{D_t})$ is path connected, we can therefore find a path $\gamma_0=\{J_t\}$  such that  $J_t \in B\cap \J_{D_t}$ and with endpoints $J_0\in B\cap \J_{D\cup D_0}$ and $J_1\in B\cap \J_{D\cup D_1}$.  The $VFC_{D_t}^{J_t}$ is well defined and becomes constant in $t$ after inclusion in $\cHH_*(\oM^{\gamma_0}(X), \Q)$

Moreover,  since the endpoints $J_0\in B\cap \J_{D\cup D_0}$ and $J_1\in B\cap \J_{D'\cup D_1}$ then Proposition \ref{9.differentD} applies to give
\best
VFC_D^{J_0}= VFC_{D\cup D_0}^{J_0}= VFC_{D_0}^{J_0}\quad \quad 
VFC_{D'}^{J_1}= VFC_{D'\cup D_1}^{J_1}= VFC_{D_1}^{J_1}
\eest
Next, $B \cap\J_{D}$ is path connected, so we can find a path $\gamma_1$ in $B \cap\J_{D}$ from $J$ to $J_0$ on which the $VFC_D$ is well defined and moreover 
\best
VFC_D^{J'}=VFC_D^{J_0}
\eest
after inclusion in $\cHH_*(\oM^{\gamma_1}(X), \Q)$. Similarly, we can find a path $\gamma_2$ in 
$B \cap\J_{D'}$ from $J_1$ to $J'$ such that  after inclusion in $\cHH_*(\oM^{\gamma_2}(X), \Q)$
\best
VFC_D^{J_1}=VFC_D^{J'}
\eest
Now let $\gamma$ denote the concatenation of the path $\gamma_1\# \gamma_0\# \gamma_2$, which is a path in $B\cap \J(X, V)$ from  $J$ to $J'$. After inclusion into $\cHH(\oM^{\gamma}(X), \Q)$, combining the last three displayed equation above then gives the desired equality
\best 
VFC_{D}^J=VFC_D^{J_0} = VFC_{D_0}^{J_0} = VFC_{D_1}^{J_1}= VFC_{D'}^{J_1}=  VFC_{D'}^{J'}
\eest
 in  $\cHH(\oM^{\gamma}(X), \Q)$. 
\end{proof}

\vspace{1cm}


\setcounter{equation}{0}
\section{Defining the VFC via stabilizing divisors} 
\label{section10}
\medskip

We can  now, at last,  complete the proof of Theorem~\ref{VFCfinecase} by  defining a virtual fundamental class for any pair $(X,V)$.  The context remains the same:  $X$ is a closed symplectic manifold and $V\subset X$ is a (possibly empty) normal crossing divisor for some compatible structure $(\w,J, g)$ on $X$.   We will also fix an energy level $E$ and work with data $(A,g)$ below energy $E$.   Finally, we assume that  
$(X, V)$ satisfies Assumption \ref{dim assumption}. 

\begin{prop}
\label{main.thm}
Let $V$ and $(X, \om, J)$ be as above. Then as long as $E(A, g)\le E$:
\begin{enumerate}[(a)]
\item For any shrinking sequence $B_i$ of $C^0$ balls in $\J(X,V)$ centered at  $(\om, J)$ there exists stabilizing divisors $D_i$ for $(X,V)$ on $B_i$ and a sequence $\ep_i\ra 0$ such that \begin{enumerate}[(i)]
 \item for generic $(\om, J_i, \nu_i)\in \JV_{D_i}(X, V)$  with $(\om,J_i)\in B_i$ and $|\nu_i|<\ep_i$, 
 the symmetrized relative moduli spaces 
 $$
 \ov \M^{J_i,\nu_i}_{A, g, n, s_V\cup [1]}(X, V\cup D_i)  
 $$
 defined as in \eqref{symmetrizedMS},   are  domain-stable and carry  fundamental classes in \Cech homology.
\item After forgetting $D_i$, their actual fundamental cycles pass to the limit to define a  rational  \v{C}ech homology  element
\bear\label{pass.lim}
\ma \lim_{i\ra \infty}\phi_{D_i*} [ \ov \M^{J_i,\nu_i}_{A, g, n, s_V\cup [1]}(X,V\cup D_i)] \in 
 \cHH_*(\ov \M^{J}_{A, g, n, s_V}(X, V);\; \Q).
\eear
\end{enumerate}
\item The class   (\ref{pass.lim}) is independent of the choices made in (a) and is invariant under deformations of the compatible pair $(\om, J)$ inside $\J(X,V)$. 
\end{enumerate}
\end{prop} 
\non {\bf Proof}. For this argument we are working in the neighborhood of a fixed moduli space  $\ov\M^J(X,V)=\ov\M^J_{A, g, n, s}(X, V)$; in particular $A, g, n, s$ will be fixed, where $E(A, g)\le E$.  

Consider the sequence of $C^0$-balls $ B_n=B(J,\tfrac{1}{n}) $. For each $n$, by Lemma  8.3 (b) there exits  an $E$-stabilizing divisor $D_n$ for $(X, V)$  on $B_n$ thus  
$\J_{D_n}(X, V)\cap B_n\ne \emptyset$. Pick $J_n\in \J_{D_n}(X, V)\cap B_n$ and consider the class 
$VFC_{D_n}^{J_n}$ defined in Definition \ref{9.defnVFC.1}. 

Now use Theorem~\ref{cor.D=D'} to construct paths $\gamma_n$ in $B_n$ from $J_n$ to $J_{n+1}$, and let  $K_m$ be the union of $J_0$ with the paths $\gamma_n$  for all $n\ge m$, as in the proof of Theorem \ref{5.CechcycleLemma}. Then under the inclusion 
$K_n \hookrightarrow K_m $: 
\best
VFC_{D_n}^{J_n}=VFC_{D_m}^{J_m}\in \cHH_*(M_m, \Q)
\eest
for all $n\ge m$, where $M_m$ is the moduli space $\ov\M(X, V)$ over $K_m$. This means that the VFC passes to the limit to give a class in 
$\cHH_*(\oM^{J}(X, V), \Q)$. Independence of the sequence $D_n, J_n$ is proven by the ladder argument: if $(D_n, J_n)$ and $(D_n', J_n')$ are two such sequences, Theorem \ref{cor.D=D'} provides  paths $\de_n$ in $B_m$ that together with the original paths form the ladder. 

This shows that (\ref{pass.lim}) is well defined, independent of choices made. This class is invariant under deformations of $(\om, J)$ in the sense that for any smooth path $\gamma$ of compatible pairs $(\om_t, J_t)\in \J(X, V)$ the two VFC have equal inclusions into the  \v{C}ech homology of the moduli space over the path: 
\best
VFC_{J_0}=VFC_{J_1}\in \cHH_*( \ov\M_{A, g, n,s}^{B}(X, V); \Q)
\eest
This follows from a similar continuity argument using the fact that $\gamma$ is compact, by taking 
$B_n$ a finite cover of it with balls of radius $1/n$ centered at points in $\gamma$. 
\qed 
\bigskip

 With $(X,V) $ and $E\ge E(A, g)$ as above, consider a shrinking sequence $B_i$ of $C^0$ balls  around  $(\om, J)$ in  $\J(X,V)$,  and pick  $E$-stabilizing divisors $D_i$ for $(X, V)$ on $B_i$.
 
\begin{defn} 
{\em The virtual fundamental class} of $\ov\M^{J}(X,V)$ is  the rational  \v{C}ech homology class
\bear\label{D.virX=virXD}
[\ov\M_{A, g, n, s} ^J(X, V)]^{vir}\quad\ma =^{def} \,
\lim_{i\ra \infty}\phi_{D_i*} [\ov\M_{A, g, n, s\cup [1]}(X, V\cup D_i)]^{vir}\in \cHH_*(\ov\M^J(X,V);\Q)
\eear 
and the corresponding   {\em relative GW invariant}  is 
\bear
\label{10.GWdef}
GW_{A, g, n, s} (X, V)=se_*[\ov\M_{A, g, n,s} ^J(X, V)]^{vir}\in H_*(\oM_{g, n+\ell(s)} \ti N_s V)
\eear
\end{defn} 

By Proposition~\ref{D.virX=virXD}b, the virtual fundamental class is well defined, independent of the choices made in its construction.  Thus we have completed the proof of Theorem~\ref{VFCfinecase} of the introduction.  In particular, in terms of GW invariants, we have the following:
\begin{cor}\label{T.GW=semipos}
The relative invariant  (\ref{10.GWdef}) is well-defined and  invariant under deformations. 
When $X$ is semi-positive and $V$ is empty, it agrees with the invariant
$GW(X)$  defined in \cite{rt2}, and when $V$ is smooth it agrees with the
relative invariant of $GW(X,V)$  defined in \cite{IP1}. 
\end{cor}

\vspace{1cm}


\setcounter{equation}{0}
\section{Independence of the stabilizing divisor }
\label{S.contrib}

\medskip

It remains  to show that the VFC of Definition~\ref{9.defnVFC.1} is independent of the stabilizing divisor $D$, and agrees with the VFC defined by Theorem~\ref{7.Maintheorem} for domain-stable  moduli spaces.  These two facts follow from Proposition~\ref{Prop11.2}  and Proposition~\ref{Prop11.1}, respectively.

Fix a homology class $A\not=0$ and a genus $g$ and set $E=E(A,g)$. Also fix a normal crossing divisor $V$  (possibly empty), and an $E$-stabilizing divisor $D$ for $V$.  Recall from Proposition~\ref{M.s-stable.V} that there is  a non-empty set of $J$ in $\J_D(X,V)$ for which there are no  non-constant $J$-holomorphic maps  into $D$  with energy  less than or equal to $E$, and the compactified moduli space  $\ov{\M}^J_{A,g,s\cup [\bf 1]}(X,V\cup D)$ is  domain-stable.  For each such $J$, under Assumption~\ref{dim assumption},  the following proposition applies, showing that the virtual fundamental classes of the stabilized and unstabilized moduli spaces agree whenever the unstabilized moduli space is already domain-stable.



\begin{prop}
\label{Prop11.1} 
Fix $A\not= 0$ and  a normal crossing divisor $V$.   Assume that $D$ is a smooth divisor transverse to $V$, and that $J\in\J_D(X,V)$ is such that
\begin{enumerate}[(a)]
\item the moduli space $\ov\M^{J}_{A, g, n, s_V}(X, V)$ is domain-stable,
\item there are no  non-constant $J$-holomorphic maps into $D$  with energy less or equal to $E(A, g)$, and 
\item either $g\le 1$ or every stratum of $V\cup D$ is at least 8 dimensional.  
\end{enumerate}
Then for each  $J\in\J_D(X,V)$ is in the set described in Proposition~\ref{M.s-stable.V}
\bear
\label{vfc.V=vfc.VD}
[\ov\M_{A, g, n, s_V}^J(X, V)]^{vir}= \phi_{D*} [\ov\M_{A, g, n, s_V\cup [{\bf 1}]}^J(X,V\cup D)]^{vir}\in 
\cHH_*(\ov\M^J_{A, g, n, s_V}(X, V); \Q)
\eear 
where $\phi_{D*}$ is  the map in diagram \eqref{9.projection}.  
\end{prop}

\begin{proof} 
First consider the case where $V=\emptyset$.  Because $D$ is $J$-holomorphic, the moduli space $\ov\M_{A, g, n}^J(X)$ of stable maps has a refined stratification. As described in the appendix, the strata are indexed by decorated graphs, whose vertices, which  correspond to irreducible component $C_i$ of $C$,  are  labeled  by a  genus, number of marked points, and the homology class $A_i=[f(C_i)]$; these are refined by also including labels indicating which $C_i$ are mapped into $D$, and, on the components not mapped to $D$, labels recording their contact  multiplicities to $D$.

The {\em principal stratum}  $\N_2$ is the open subset of $\M_{A, g, n}^J(X)$ consisting of maps $f:C\ra X$ with smooth domain that are transverse to $D$ and such that $f^{-1}(D)$ is disjoint from the special points of $C$.    By marking all points in $f^{-1}(D)$, one obtains a subset $\tilde{\N}$ inside  the relative moduli space $\M_{A, g, n, ({\bf 1})}^J(X,D)$ with a forgetful map $\tilde{\N}\to \N_2$.  The symmetric group $S_\ell$ acts by permuting the added $\ell$ marked points, and $\N_2$ is homeomorphic to the subset $\N_1=\tilde{\N}/S_\ell$, which is  the principal stratum   of  the  quotient   space $ \ov\M_{A, g, n, [{\bf 1}]}^J(X,D)$ defined by \eqref{symmetrizedMS}.

The same  stratification scheme applies  to the moduli spaces obtained by replacing $J$ by $(J, \nu)$,  where $\nu$ is any perturbation whose normal component to $D$ vanishes along $D$, and that is  invariant under the $S_\ell$ action.  These conditions are satisfied if $\nu$ is pulled back from a Ruan-Tian perturbation on $\ov \U_{g, n}^G\ti X$ whose normal component vanishes on   $\ov \U_{g, n}^G\ti D$. For every such pair $(J, \nu)$, the  map that forgets the  divisor $D$ and the extra $\ell$ points (and collapses any unstable domain components that are mapped to points) is a continuous map
\bear
\label{11.2}
\phi_D: \ov\M_{A, g, n, [{\bf 1}]}^{J,\nu}(X,D) \ra \ov\M_{A, g, n}^{J,\nu}(X)
\eear
that restricts to a homeomorphism from  $\N_1$ to $\N_2$.  The spaces in \eqref{11.2}  can be regarded as  compactifications of $\N_1$ and $\N_2$ respectively. So it suffices to show that there exists a sequence of parameters $\nu$ converging to zero such that:
 \begin{enumerate}[(i)] \setlength\itemsep{4pt}
\item All maps in the  principal strata $\N_1$ and $\N_2$ of the spaces in \ref{11.2} are regular, and  hence $\N_1$ and $\N_2$ are  oriented manifolds of dimension $\iota$. 
\item All the other strata of each compactification are  manifolds of dimension at most $\iota-2$.  
\item The restriction of $\phi_D$ to the principal stratum $\N_1$ preserves orientation.
\end{enumerate}   
It  then follows that the \Cech homology map induced by \eqref{11.2} takes the virtual fundamental class of the right-hand space to the virtual fundamental class of the left-hand space (see \cite{IPFC} for details).

\medskip

But facts (i)-(ii) are precisely what Tehrani and Zinger  proved in \cite[\S3.1]{tz}. Their key observation is that one can use condition (c) and a dimension count to  show that, for certain $(J,\nu)$,  the compactified moduli space contains no maps that have genus $g\le 2$ components mapped into $D$.  The dimension counts in \cite[\S3.1]{tz} imply  that under assumptions (b) and (c), the actual dimension of each stratum except the principal stratum  is still  at most $\iota-2$.  

The proof involves considering the restrictions of maps $f:C\to X$ to those irreducible components  $C_i$ mapped into $D$.  Note that, for each $(J, \nu)$ as above, the regularity of $f_i:C_i\to D$ as a map into $D$ is not the same as    the regularity of $f_i$ as a map into $X$.  The difference is captured by the normal component of the linearization, whose index,   under assumption (b), equals $2(1-g)$, which is non-negative for $g=0, 1$.  From this fact and dimension counts, 
 Tehrani and Zinger  establish  (ii) above.  We refer the reader to the proof of Theorem 1 in \cite[\S3.1]{tz} for   the required transversality and orientation arguments, including extensions without assumption (a).
 
\bigskip

The argument above extends to the case when $V$ is a normal crossing divisor, now using the stratification of $X$ induced by $V$ as in the proof of Proposition~\ref{M.s-stable.V}. 
In fact, under assumptions (a)-(c), one does not even need to consider the sections of the normal bundle  to $D$ that arise from renormalization in order to construct a compactification of $\N_1=\M_{A, g, n, s_V, [{\bf 1}]}(X, V\cup D)$ with codimension 2 boundary strata. On the other hand,  as above, $\N_1$ is identified with
the principal stratum $\N_2$   of $\M_{A, g, n, s_V}(X, V)$.  Again, dimension counts extending those in \cite[\S3.1]{tz}  show that,  under assumptions (a)-(c),  all strata in the boundary  $\oM_{A, g, n, s_V}(X, V)\setminus \N_2$ have codimension at least 2.
\end{proof}

\medskip

\begin{rem} 
Without assumption (c),   the two virtual fundamental cycles cannot be related using dimension counts alone. In general, an analysis of \eqref{11.2} shows that
\best
[\ov\M_{A,g, n, s}(X, V)]^{vir}&=&  \phi_{D*} [\ov\M_{A, g, n, s, [1]}(X, V\cup D)]^{vir} + \mbox{ correction terms, }
\eest
where the correction terms come from the contribution of    strata of $\oM(X)$ consisting of maps with components in $D$ whose dimension is equal to, or higher than, the dimension of the virtual fundamental class.  There are several ways to calculate these contributions, but all involve gluing, either  using the trivial decomposition  of $X$ as the union of $(X, D)$ and the projectivized normal bundle  $\P_D$ (with $D$ identified with the infinity section of $\P_D$),  or alternatively using an obstruction bundle gluing argument as in \cite{HT}. 
\end{rem}

\begin{prop}
\label{Prop11.2}
 Assume $V$ is a normal crossing divisor in $(X, \om, J)$, while $D$ and $D'$ are two smooth $J$-holomorphic divisors such that $V, D, D'$ are in general position.  Fix $A\ne 0$ and assume that
\begin{enumerate}[(a)]\setlength\itemsep{4pt}
\item  the moduli space $\ov\M^{J}_{A, g, n, s\cup[{\bf 1}]}(X, V\cup D)$ is domain stable, and
\item  there are no  non-constant $J$-holomorphic maps into $D$ or $D'$  with energy less or equal to $E(A, g)$. 
\item either $g\le 1$ or every stratum of $V\cup D\cup D'$ is at least 8 dimensional.  
\end{enumerate}
Then $\ov\M^{J}_{A, g, n, s\cup [{\bf 1}] \cup [{\bf 1}]}(X, V\cup D\cup D')$ is also domain stable,  and  under the forgetful map
 \best
 \phi_{D'}: \ov\M^{J}_{A, g, n, s\cup [{\bf 1}] \cup [{\bf 1}]}(X, V\cup D\cup D')\ra \M_{A, g, n, s\cup [{\bf 1}]}^J(X,V\cup D)
 \eest 
 defined by \eqref{9.projection}  with $V$ replaced by $V\cup D$ and $D$ by $D'$, we have the equality
 \bear
\label{vfc.V=vfc.VD}
\phi_{D'*}[\ov\M_{A, g, n, s,\cup [{\bf 1}] \cup [{\bf 1}]}^J(X, V\cup D\cup D')]^{vir}= [\ov\M_{A, g, n, s\cup [{\bf 1}]}^J(X,V\cup D)]^{vir}
\eear 
as elements of $ \cHH_*(\ov\M^J_{A, g, n, s_V \cup [{\bf 1}]}(X, V\cup D); \Q)$. 
\end{prop}
\begin{proof} The assumptions immediately imply the first assertion because any constant components in $D'$ or $D\cap D'$ must have stable domains. Therefore 
each of the moduli spaces that appear in \eqref{vfc.V=vfc.VD}  has a well-defined virtual fundamental class,  obtained by using Ruan-Tian perturbations. However, the Ruan-Tian perturbations needed to achieve transversality are {\em a priori} different for these two moduli spaces. To relate the two  virtual fundamental classes, we must  restrict to a common class of perturbations for which the forgetful map 
 \best
 \phi_{D'}: \ov\M^{J, \nu}_{A, g, n, s\cup [1]}(X, V\cup D\cup D') \ra 
 \ov\M^{J, \nu}_{A, g, n, s}(X, V\cup D) 
 \eest 
 is well-defined and continuous. This is the case, for example, when we restrict to the subset of  $(J, \nu)\in \JV_D(X, V)$ for which the normal component to $D'$ of the perturbation $\nu$  vanishes.   It is straightforward to verify that this subset is a non-empty Banach submanifold of $\J_D(X,V)$.
 
It then suffices to find a sequence of such perturbations $\nu$ converging to zero, and open subsets 
\best
\N_1 \subseteq \M^{J, \nu}_{A, g, n, s\cup [1]}(X, V\cup D\cup D') \quad\mbox{ and }  
\N_2 \subseteq \M^{J, \nu}_{A, g, n, s}(X, V\cup D)
\eest
with the following properties: 
\begin{enumerate}[(i)]
\item each $\N_1$ and $\N_2$ consist only of regular points of their corresponding moduli space, and therefore are oriented manifolds of dimension $\iota$.  
\item $\phi$ restricts to an orientation preserving homeomorphism between $\N_1$ and $\N_2$
\item the complement of $\N_1$ in 
$\ov\M^{J, \nu}_{A, g, n, s\cup [1]}(X, V\cup D\cup D')$ is a set of dimension $\iota-2$. 
\item the complement of $\N_2$ in 
$\ov\M^{J, \nu}_{A, g, n, s\cup [1]}(X, V\cup D)$ is a set of dimension $\iota-2$. 
\end{enumerate}
As above, the sets $\N_1$ and $\N_2$ can be chosen to be the principal strata of the refined stratifications of the two moduli spaces, corresponding to maps $f:C\ra X$ with smooth domain, transverse to $D$ and $D'$, not entirely contained in $V\cup D\cup D'$ and missing the singular locus of $V\cup D\cup D'$. The only difference between $\N_1$ and $\N_2$ is that for $\N_1$ all the points in $f^{-1}( V\cup D\cup D')$ are already marked (and decorated with the multiplicity of contact), while for $\N_2$ only the points in $f^{-1}(V\cup D)$ are marked and decorated.   
  
The considerations of \cite[\S 3.1]{tz} now easily extend. Again, the key observation is that  condition (c) implies that all irreducible components that are mapped to  $D$, $D'$ or $D\cap D'$ have genus $g\le 1$.  For these,
 the index of the normal operator is nonnegative by (b), and therefore the actual dimension of each stratum of either compactification, except the principal stratum,   is again  at most $\iota-2$. 
 \end{proof}

 \vspace{1cm}
\setcounter{equation}{0}
\setcounter{section}{0}
\setcounter{theorem}{0}
{
\renewcommand{\theequation}{A.\arabic{equation}}
\renewcommand{\thetheorem}{A.\arabic{theorem}}
\renewcommand{\thesubsection}{\bf A.\arabic{subsection}}
\appendix{}

\section*{A{\sc ppendix}}
\medskip

This appendix establishes several  results  used  in previous sections that do not appear in the literature.  These are basic facts  about the topology of the moduli space of $J$-holomorphic maps  and the existence of $V$-compatible almost complex structures.

\subsection{Proof of Theorem~\ref{3.TopTheorem}}

Theorem~\ref{3.TopTheorem} states several properties   about the topology of the moduli space.    The proofs are presented below,  organized into four steps.  Along the way,  two general facts about metric spaces and maps $f:X\ra Y$ are used repeatedly:
\begin{enumerate}[(i)]
\item  If $(Y, d)$ is a metric space, then  $f:X\ra Y$  induces a pseudo-metric $f^*d$ on $X$; this defines a topology on $X$ for which $f$ is continuous.
\item  A metric on $X$ induces a  metric $d_{\cal H}$ -- the Hausdorff  distance --  on the set $Z=Subsets_{c}(X)$ of  its non-empty compact subsets that is compact whenever $(X,d)$ is compact.
 \end{enumerate}

 \medskip

\noindent{\sc Step 1.}    Using Facts~(i) and (ii), define an initial topology on $\ov\M_{A, g, n}(X)$ by pulling back the metrics on $\J(X)$, $\ov\M_{g,n}$ and $\ov\M_{g, n+1}\ti X$ by 
 $\pi\ti \st \ti \Gamma$. More precisely, define the pseudo-metric $d_0$ by setting
\bear
\label{3.Hdist}
d_0\big((C,f, J), (C',f',J')\big) \ =\ d_{\J}(J, J')\ +\   d_{\oM_{g, n}}\big(\st(C),\st(C')\big)\ +\ d_{\cal H}(\Gamma_f, \Gamma_{f'})
\eear
with $\Gamma_f$ as in (\ref{3.defGammaf}).

\begin{lemma} \label{L.d0=0}Assume $f, h$ are stable (perturbed) pseudo-holomorphic maps with same topological data 
$(A, g, n)$. If  $\Aut\; C_f=1$,  then $d_0(f, h)=0$ if and only if $f$ and $h$ differ by a reparametrization. 
\end{lemma}
\begin{proof} One direction is clear because $d_0$ is reparametrization invariant. Conversely,  assume $d_0(f, h)=0$. Then $\st(f)=\st(h)\in \ov\M_{g, n}$, $\pi(f)=\pi(h)$ and $\Gamma_f=\Gamma_h$. By assumption,  the graph $\Gamma_f$ is an embedded nodal curve in $\ov\M_{g, n+1} \ti X$.   Because $C_f$ is  already stable with $\Aut\; C_f=1$ and $\st(h)=\st(f)$ so we can  assume, after precomposing $f$ by an isomorphism, that  $C_f$ is obtained from  $C_h$ by collapsing the unstable rational components $C^{+}_h$ of $C_h$; moreover, $C_f=C_f/\Aut C_f$ is canonically isomorphic to the fiber of the universal curve over $[C_f]\in \ov\M_{g, n}$.  Since $\Gamma_f=\Gamma_h\subset \ov\U_{g, n}\ti X$ then the restrictions of $f$ and $h$ to the stable part $C_f$ of their domain must now be equal. The energies then satisfy $E(h)= E(f)+E(h^{+})$ where $h^{+}$ is the restriction of $h$ to $C_h^{+}$.  On the other hand,  $E(f)=E(h)$ because $f$ and $h$ represent the same homology class. Since $h$ is a stable map, $C^+_h=\emptyset$ 
since otherwise $E(h^+)>0$.  We conclude that $f$ and $h$  have the same domain and  $f=h$. 
  \end{proof}

\medskip

\noindent{\sc Step 2. }  Define a sequence of pseudo-metrics on $\ov\M_{A, g, n}(X)$ by  setting $d_k(f,f')= d_{\cal H}\big( \Gamma^k_{f}, \Gamma^k_{f'}\big)$ and then let
\best
d(f, f')\ =\    d_0(f,f')\ +\ \sum_{k=2}^\infty \ 2^{-k}\ \frac{d_k(f,f')}{1+d_k(f,f')}
\eest 
\begin{lemma} 
\label{A.dmetric}
$d$ is a metric on $\ov\M_{A,g, n}(X)$. 
\end{lemma}
\begin{proof} Assume by contradiction that  $f, h\in \ov\M_{A,g, n}(X)$ such that $d_k(f, h)=0$ for all $k$ but $f$ is not a reparametrization of $h$. We can then add a finite collection of marked points $x$ to the domain of $f$ to ensure that its domain is stable with trivial automorphism group. Then $d_0(\Gamma_{(f, x)}, \Gamma^k_h)\leq d_k(f,h)=0$, and so $d_0((f, x), (h, y))=0$ for some 
collection $y$ of marked points.  But then by the previous lemma $(f, x)$ and $(h, y)$ must differ by a reparametrization, contradiction. \end{proof}

\medskip

\noindent{\sc Step 3. }   After symmetrizing the metric by the actions of the finite groups $G$ and $S_n$, we may assume that $d$ is invariant under the action of $G\ti S_n$.   It follows that  all  maps $\pi$, $se$,  $\phi_k$ and $\Gamma^k$ above are continuous.
 
\medskip

\noindent{\sc Step 4. }  The proof is completed by using  the following version of the Gromov Compactness Theorem.

\begin{theorem}
\label{A.GCT}
 Every sequence $\{f_n:C_n\to X\}$ of $J_n$-holomorphic maps with fixed arithmetic genus and number of marked points,  uniformly bounded energy, and with $J_n\to J$ in $\J^\ell$  has a subsequence that converges in $C^0$ up to reparameterization to a $J$-holomorphic map $f:C\to X$.
\end{theorem}

Because $C^0$ convergence implies the convergence of the graphs $\Gamma^k_{f_n}$, this immediately shows that $\pi: \oM^{G, E}_{g,n}(X)\ra \J^\ell$  is proper.   The maps  $\phi_k$ in  (\ref{3.naturalmaps}) are also proper because their fibers are closed subsets of $C^k$.   Furthermore,  $\phi_k$ is perfect:  it is continuous, surjective  and, by an observation of Palais \cite{pal}, a  proper continuous map to a metric space is closed.  Thus all parts of Theorem~\ref{3.TopTheorem}  hold.
\qed

\bigskip

Rather surprisingly, convergence in the metric of Theorem~\ref{3.TopTheorem} implies convergence in the Gromov topology with higher regularity.  For a precise statement, a brief diversion is necessary to clarify what convergence with higher regularity means in the present context.

Suppose  $f_m\in\ov\M^E(X)$ satisfies $d(f_m, f_0)\ra 0$.  Then, as in the proof of Lemma \ref{A.dmetric} we can add marked points to the domain of $f_0$ to  obtain an $n+k$-marked  stable curve $C_0$ with $\Aut\; C_0=1$.  The assumption that $d^k(f_m,f_0)\to 0$ then means that one can choose $k$ marked points to make the domain of each $f_m$ an $n+k$ curve $C_m$.  By the semicontinuity properties used in the proof of Lemma~\ref{JstableOpenLemma}, the $C_m$ are stable with $\Aut\; C_m=1$ for large $m$.  We can then regard the maps $f_m$ and $f_0$ as maps defined on  fibers of the universal curve over  an open neighborhood $U$ of $[C_0]$ in $\ov\M_{g,n+k}$.

Now for each compact set $K$ in $C_0\setminus \{nodes\}$ we can find  a fiber-preserving biholomorphic  map
$$
\tau:    K\ti U \to  \ov\U_K
$$
where   $\ov\U_K$ is the  portion of the universal curve over $U$ and outside a fixed open neighborhood of the nodal set ${\cal N}\subset \ov\U$.  Then $g_m=f_m\circ \tau:K\to X$ is a sequence of maps with the same domain.

 Elliptic regularity results now apply.  In particular,  if $\{f_m\}$ is a sequence of $J_m$-holomorphic maps from $K$ to $X$ with $f_m\to f_0$ in $C^0$ and $J_m \to\J_0$ in $W^{\ell,p}$ for $\ell\geq 1$ and $\ell p>2$,   then $f_m\to f_0$ in $W^{\ell,p}$ \cite{is}.   This applies for all $\ell\geq 1$ and all compact sets.

  Letting $\J^\infty(X)$ denote the space $\bigcap_\ell \J^\ell(X)$ of smooth tame structures on $X$, we conclude that if $J_m\to J_0$ in $C^\infty$ then $f_m\to f_0$ in $C^0$, in  Hausdorff distance  and in $C^\infty$ on every compact set $K\subset C_0\setminus \{nodes\}$, which is the standard definition of Gromov convergence used in the literature.  Thus we have:
 
 \begin{cor}
 \label{A.cor}
The topology on $\ov\M^E(X)$ over $\J^\infty(X)$ given by Theorem~\ref{3.TopTheorem}  is equivalent to the usual Gromov topology.
 \end{cor}

 It is straightforward to extend these results to  decorated and relative versions of the moduli spaces.

\bigskip

\subsection{The stratification of the universal moduli space}

The universal moduli space $\ov \M(X)$ is stratified by the topological type $t$ of its elements $(f, C)$.  As in \cite{ms2}, each type can be represented by the dual graph whose vertices, which  correspond to irreducible component $C_i$ of $C$,  are  labeled  by a  genus, number of marked points, and the homology class $A_i=[f(C_i)]$. For relative moduli spaces, one also keeps track of the contact information to $V$.  The types are partially ordered:  $t \prec t'$ if $t'$ is obtained from $t$ by smoothing a node.

Each  graph type $t$ defines a open  stratum $\M_t$ of the universal moduli space, and we set $\oM_t=\bigcup_{\tau\preceq t} \M_{\tau}$.    Fix $E$ and let  $\cal T_{U,E}$ denote the collection of all topological types that are represented by maps $f$  in the universal moduli space $\ov\M^{U,E}(X)\to U$ over $U\subset \J$ with  $\w(A)\le E$ and with $3g-3+n \le E$. 
Gromov compactness implies $\cal T_{U, E}$ is finite for any $U$ with compact closure in $C^0$. It also implies the following fact.
\begin{lemma}
\label{L.top.not.jumps} 
Each $J\in \J(X)$  has a $C^0$ neighborhood  $U$ in which
\best
\cal T_{U, E}= \cal T_{J, E}
\eest
In particular, every triple $(A, g, n)$ represented by a map in $\ov\M^{U, E}(X)$ is already represented by one in  $\ov\M^{J, E}(X)$.  
\end{lemma}

\begin{proof}
The inclusion $\cal T_{J, E} \subset \cal T_{U, E}$ is obvious. For the other inclusion, assume the contrary: that there is a $C^0$-convergent sequence   $J_k\ma\ra J_0$ and maps $f_k:C_k\ra X$  in $\oM_t^{J_k, E}(X)$ for some $t$ with $\oM_t^{J, E}(X)$ empty.  The inequality $3g_k-3+n_k\le E$ gives a uniform bound on the topological type of $\st(C_k)$. We also have the energy   bound $E(f_k)=\w(A)\le E$ and  a uniform lower bound $\al>0$ on the energy of each non-constant $J_k$-holomorphic sphere. Since $f_k$ is a stable map,  this uniformly bounds  the number of unstable rational components of $C_k$. Hence there are finitely many possible domain types  for  $C_k$ (unstable genus 1 components  can occur only if $C_k$ is smooth, genus 1 with no marked points).  

 By Gromov compactness a   subsequence of the  $f_k$ converges in $C^0$, in homology,  and in energy to  a $J$-holomorphic stable limit $f_0:C_0\ra X$ that is in $\oM_t^{J,E}$; contradiction.    
\end{proof} 

\bigskip

 \subsection{Adapting $J$ to a normal crossing divisor}

 In Section~8 we introduced a stabilizing divisor $D$, replaced the normal crossing divisor $V$ with $V\cup D$, and replaced $J$ by  a nearby $J'$  compatible with $V\cup D$.  This subsection explains how to find such a $J'$.  It begins at the level of linear algebra.

Fix  a vector space $W$  of dimension $2n$  with a hermitian structure $(\w, J, g)$.   The Grassmann manifold $Gr_{2\ell}$ of  codimension~$2\ell$ subspaces  of $W$ is compact and has  a  canonical  Riemannian metric on $Gr_{2\ell}$ induced by the metric $g$ on $W$. Furthermore, 
the subset  $Gr_{2\ell}^J$  of the $J$-invariant subspaces is a submanifold.

  \begin{defn} 
We say that    $V \in Gr_{2\ell}$ is {\em $\ep$-holomorphic} if  it  lies in the $\ep$-tubular neighborhood of $Gr_{2\ell}^J$.
\end{defn} 
Similarly, an ordered configuration $V= \{ V_i\}$ of $k$  codimension 2 linear subspaces of $W$ is a point in the product $G_k(W)=Gr_{2}\times\cdots\times Gr_{2}$ and  $V$  is $\ep$-holomorphic if it lies in the $\ep$-neighborhood of the submanifold 
\bear
\label{A.GrassmannProduct}
G^J_k(W)\ =\ Gr^J_{2}\times  \cdots \times Gr^J_{2}\ \subset \ G_k(W)
\eear
of  $J$-invariant configurations.    It is also useful to consider the singular variety $S\subset G_k(W)$ of configurations not in general position, and say that
 $V$ is {\em $\alpha$-general} if $\mbox{dist}(V,S)\ge \alpha$.  Observe that, because $G_k^J(W)$ is a submanifold, there are constants $c_1, \ep_0>0$ so that if $V$ is an $\ep$-holomorphic  configuration with $\ep<\ep_0$ then there is a $J$-invariant configuration $V'$ within distance $c_1\ep$ of $V$, and in fact $V$ is also $\omega$-symplectic. Moreover, if $V$ is $\al$-general and $\ep$ is sufficiently small (depending on $\al$) then the complex configuration $V'$ is $\al/2$-general.

It is useful to think of  configurations as assembled from 2-dimensional subspaces in the following way. 

 \begin{defn} Let $V= \{ V_i\}$  be a configuration of $k$ codimension 2  subspaces of $W$ in general position and with common  intersection $V_{all}=\cap_i V_i$.    An  {\em adapted splitting}  of $V$ is $V_{all}$ together with a set of $k$  2-dimensional complementary subspaces $N_i$  such that   
  \bear\label{V.split}
W= V_{all}\oplus \oplus_{i}  N_i
\eear 
and  for any $I\sqcup I'\subset\{ 1, \dots, k\}$, 
 \bear\label{split.W}
 \ma\cap_{i\in I} V_i\ =\ (\ma\cap_{i\in I\sqcup I'} V_i)\oplus \ma\oplus_{j\in I'}N_j.
\eear
\end{defn} 

Adapted splittings  always exist. For example, one can choose an invertible linear transformation $L:W\to \cx^n$ such that $L(V)$ is the standard configuration $V^0$ of the first $k$ complex coordinate hyperplanes $\{x^i=0\}$ in $\cx^n$ and take $N_i$ to be the inverse image of  the complex line in the $i$ direction (these complex lines form the standard splitting $N^0$ of $V^0$).   An adapted splitting (\ref{V.split}) is called $J$-invariant if $V_{all}$ and each $N_i$ are $J$-invariant;  it follows that each $V_i$ and all intersections (\ref{split.W}) are $J$-invariant.   
Thus there is a dual perspective:  given a general point $N=(V_{all}, N_1, \dots, N_k)$ in the 
Grassmann
$$
\widehat{G}_k(W)\ =\ Gr_{2k}\times Gr_{2n-2}\times \cdots \times Gr_{2n-2}
$$
we obtain a configuration $V=\{V_i\}$ where
\bear
\label{A.constructVi}
V_i\ =\ V_{all} \cup \underset{j\not= i}{\oplus}   N_j.
\eear
Again $\widehat{G}_k(W)$ contains  a submanifold $\widehat{G}^J_k(W)$ of $J$-invariant subspaces and a variety $\widehat{S}$ of splittings not in general position, and we say $N\in \widehat{G}_k(W)$  is $\ep$-holomorphic if it lies in the $\ep$ neighborhood of $\widehat{G}^J_k(W)$ and is $\alpha$-general if it lies outside the $\alpha$-neighborhood of $\widehat{S}$.

\begin{lemma}
\label{etatrans-ephololemma} Fix a hermitian vector space $(W, \om, J, g)$ and $\al>0$. Then there exist constants $c_\al, C_\al$ and $\ep_0>0$  with the following property: for every     $\al$-general and $\ep$-holomorphic configuration $V\subset W$ with $\ep<\ep_0$,  there exists a $V$-adapted splitting $N_V=\{N_i\}$ that is $c_\al$-general and $C_\al\ep$-holomorphic.
\end{lemma}

\begin{proof} 
Whenever $V$ is $\al$-general we can find a linear transformation $L$ that takes $V$ into the standard configuration $V^0$ with norms $|L|$ and $|L^{-1}|$  bounded by a constant  $C'_\al$ depending on $\al$ but independent of $V$. Pulling back   the standard splitting $N^0$  gives an adapted splitting $L^{-1}(N^0)$ for $V$ that is  $c_\al$-general for a constant $c_\al$ independent of $V$.  If $V$ is $J$-invariant, we can find a complex linear transformation $L$ that  pulls back $N^0$ to a $J$-invariant $c_\al$-general splitting $N_V$. Finally, if $V$ is only $\ep$ $J$-holomorphic, for $\ep$ sufficiently small (depending on $\al$)  homotoping $V$ to the nearby $J$-invariant $V'$ gives a path $L_t$ of linear transformations of length $c_1\ep$ and so a homotopy from the splitting $N_V$ to a $J$-invariant one $N_{V'}$ of length $C_\al \ep$. \end{proof}
\medskip

Now consider a compact manifold $X^{2n}$ with an almost K\"{a}hler structure
$(\w, J, g)$ and a topological normal crossing divisor $V=\{V_i\}$, that is,
assume that $V$ satisfies Definition~1.3 of \cite{i-nor} without the
requirement that $V$ be $J$-holomorphic.  In particular, the branches of $V$
are in general position and each intersection $V_I=V_{i_1}\cap\cdots\cap V_{i_k}$  is a submanifold. Also fix an adapted splitting $N_V$ (this is called the ``normal bundle to $V$'' in   \cite{i-nor}). The local model of $X$ near a point $p\in V_I$ is then described via a local diffeomorphism sending the branches $V_i$ into the $i^{\mbox{\tiny th}}$  coordinate planes in $\cx^n$ as in \cite{i-nor}, but without any compatibility conditions with $(\om, J)$. 

Along each depth $k$ stratum $V_I$ where $k$ branches of $V$ meet, the restriction of the configuration $N_V$ to  $V_I$ defines a smooth section of the Grassmann bundle $\widehat G_k(TX)$ over $V_I$  that  lies in the subbundle $\widehat G^J_k(TX)$ at each point $p\in V_I$ where the configuration $N_V$ and therefore $V$  is $J$-holomorphic. In fact, using the local models, we can extend the splitting (\ref{V.split}) over a neighborhood of $V_I$ (by extending smoothly and then projecting),  making the extension agree with the existing splitting on those  strata $V_{I'}$ whose closure contains $V_I$.    This gives sections $\sigma_I$ defined on a neighborhood of  $V_I$ for each $I$, and these are compatible:   if $I'\subset I$ then $V_{I}$ lies in the closure of  $V_{I'}$, and on the intersection of  their tubular neighborhoods there is a forgetful map  (that forgets the branches $V_i$ for $i$ in $I\setminus I'$) 
\bear
\label{A.Gkcompatible}
\widehat G_{|I|}(TX) \hookrightarrow \widehat G_{|I'|}(TX)
\eear
 which takes $\sigma_{I}$ to $\sigma_{I'}$.

\begin{prop} 
\label{0-jetlemma} 
Assume $(X, J, g, \om)$ and $\al>0$ is fixed. Then there exist constants 
$C_\al, \ep_0>0$ with the following property: for any $\ep<\ep_0$ and any topological normal crossing divisor $(V, N_V)$  in $X$ that is $\al$-general and $\ep$-holomorphic there is   an almost complex structure $J_V$ on $X$ with $|J-J_V|\le C_\al\ep$ such that
\begin{enumerate}[(a)]
\item $(V, N_V)$ is $J_V$-holomorphic, and
\item  $J_V$ is $V$-compatible in the sense of \cite{IP1} and \cite{i-nor}.
\end{enumerate}
 \end{prop}

\begin{proof}
Both $(X, V)$ and the Grassmanian manifolds are compact so we can find uniform bounds for all pointwise estimates below.  

For any $V$ which is $\ep$-holomorphic and $\al$-general, we can  deform each section $\sigma_I$ to a  section $\phi_I$  of the complex configurations space $\widehat G^J_{|I|}(TX)$  and, by  further deformations,  make the $\{\phi_I\}$ compatible under the inclusions $\widehat G_{|I'|}^J(TX) \hookrightarrow \widehat  G^J_{|I|}(TX)$ corresponding to (\ref{A.Gkcompatible}).  Because for $\ep$ small these deformations take place in a small tubular neighborhood of $\widehat  G^J_{|I|}(TX)$ in $\widehat  G_{|I|}(TX)$ (away from the singular locus $S$ of non transverse configurations),  the deformation is unique up to homotopy and $\mbox{dist}(\sigma_I, \phi_I)\le C_\al\ep$ for some uniform constant $C_\al$ (independent of $\ep$ and $V$), using the canonical metric on $\widehat  G_{|I|}(TX)$ induced by the metric $g$ on $X$.  

At each point $p$ in a neighborhood of $V$, the section $\phi$ corresponds to a $J$-complex configurations $N_\phi$ in $T_pX$ that is close to $N_V$ in the sense that the $g$-orthogonal projection 
\bear
\label{A.piphiproj}
\pi_\phi:  N_V\to N_\phi
\eear
satisfies $\|\pi_\phi-\mbox{Id.}\| \le C\ep$.  Here $\pi_\phi$ is defined by taking $TV_{I}$ and each $N_i$, for $i\in I$ onto the corresponding subspaces of $N_\phi$, therefore is well defined on  $T_pX= T_pV_I\oplus_{i\in I} N_i$.  In particular, $\pi_\phi$ is an isomorphism for small  $\ep$.   Define $J_V$ by 
$$
J_V \ =\ \pi_\phi^*J\ =\  (\pi_\phi^{-1} )_*\circ J \circ (\pi_\phi)_*    
$$
Then $J_V$ is an almost complex structure on a neighborhood of $V$ which preserves $N_V$,  and therefore preserves $V$,  and $|J-J_V|\le C\ep$ with $C_\al$ independent of $\ep$ and $V$.    

This defines $J_V$ in a neighborhood of $V$.  Statement (a) follows because we can use  the method of Theorem~A.2 of \cite{IP1} to
merge $J_V$ into $J$, maintaining the bound $|J-J_V|\le C_\al\ep$.

Finally, we  can also achieve  statement (b) by successively modifying  $J_V$, beginning with the deepest stratum $V_{all}$.   The  needed $V$-compatibility  condition along a stratum $V_I$ requires that the Nijenhuis tensor ${\cal N}$ of $J_V$ at $p\in V_I$  satisfy ${\cal N}(v, \xi)\in TV_i$ for    every $v\in TV_I$ and every $\xi$ in the normal bundle to $V_I$, which is  $N_I=\oplus_{i\in I} N_i$.  This can be achieved by applying the proof of Theorem~A.2 of \cite{IP2}, parallel transporting in  directions inside $N_I$ along $V_I$ and merging into the existing $J$.  This yields a new $J_V$ that is now $V$-compatible in a neighborhood of $V_I$ (and preserves the fact that $(V, N_V)$ is $J_V$-holomorphic everywhere). One then repeats the process along the lower strata inductively to construct the required $V$-compatible $J_V$. 
 \end{proof}
 
Proposition~\ref{0-jetlemma}   has the following corollary that was used repeatedly in Section \ref{section8}.

  \begin{cor} 
\label{0-jetlemma.D} Suppose that $V$ is a  $J$-holomorphic normal crossing divisor in 
$(X, \omega, J, g)$ and fix $\eta>0$. Then there exists constants $\ep, C_\eta>0$  with the following property: for each Donaldson divisor $D$ that is $\ep$-holomorphic for  $\ep<\ep_0$ and $\eta$-transverse to $V$ (in the sense of \cite{au2}),  $V\cup D$ is a symplectic normal crossing divisor for some $J'$ with $|J-J'|\le C_\eta\ep$.  
\end{cor}
\begin{proof} 
By compactness, $V$ is $\al$-general for some $\al>0$. Similarly, because $D$ is $\eta$-transverse to $V$,  $V\cup D$ is a topological normal crossing divisor that  is $\al$-general for some $\al=\al(\eta)>0$ independent of $D$. Lemma \ref{etatrans-ephololemma} applies pointwise to produce an adapted splitting $N_{D\cup V}$  for $V\cup D$ that   is $c_\eta$-general and $C_\eta'\ep$-holomorphic with   constants $c_\eta$ and $C'_\eta$  independent of $\ep$ and $D$.

Proposition \ref{0-jetlemma} then applies provided $\ep$ was sufficiently small (less than an $\ep_0$ depending on $\eta$), yielding  a $V\cup D$ compatible almost complex structure $J'=J_{V\cup D}$ with $|J-J'|\le C_\eta\ep$ for a constant   $C_\eta$ independent of $\ep$ and $D$.  \end{proof}

\vspace{1cm}


\vspace{8mm}



}

\end{document}